\documentclass[a4paper,USenglish,thm-restate,numberwithinsect]{lipics-v2021}
\pdfoutput=1
\usepackage[utf8]{inputenc}
\usepackage[T1]{fontenc}
\usepackage{anyfontsize}
\usepackage{mathtools}
\usepackage{amssymb}
\usepackage{amsthm}
\usepackage{thmtools}
\usepackage{babel}
\usepackage{stmaryrd}
\SetSymbolFont{stmry}{bold}{U}{stmry}{m}{n}
\usepackage{mathrsfs}
\usepackage{mdwtab}
\usepackage{multirow}
\usepackage{calc}
\usepackage{braket}
\usepackage[noadjust]{cite}
\usepackage[all]{xy}
\usepackage[defaultlines=4,all]{nowidow}
\usepackage{microtype}
\usepackage[commentmarkup=todo,commandnameprefix=always]{changes}
\usepackage{zref-clever}
\zcsetup{nameinlink,cap}
\zcRefTypeSetup{claim}{
Name-sg = Claim ,
name-sg = claim ,
Name-pl = Claims ,
name-pl = claims ,
}
\zcRefTypeSetup{conjecture}{
Name-sg = Conjecture ,
name-sg = conjecture ,
Name-pl = Conjectures ,
name-pl = conjectures ,
}
\newcommand*{\minwidthbox}[2]{%
	\makebox[{\ifdim#2<\width\width\else#2\fi}]{#1}%
}
\newcommand{\interp}[1]{\minwidthbox{$\mathsf{#1}$}{6pt}}
\newcommand{\lin}{\mathrel{\vartriangleleft}}

\DeclareMathOperator{\sparse}{Sparse}
\DeclareMathOperator{\ground}{{L}}
\DeclareMathOperator{\restr}{\rm Restr}
\newcommand{\rel}[2]{{\rm Rel}_{#2}(#1)}
\newcommand{\am}[1]{\mathop{\rm Amalg}(#1)}
\newcommand{\bij}[4]{%
\xymatrix@M=0pt{{#1}\,\ar@<2.5pt>@{|->}[r]^{\scriptscriptstyle\smash{#2}}&\,{#3}\ar@<1.5pt>@{|->}[l]^{\scriptscriptstyle\phantom{x}\smash{#4}\phantom{x}}}}
\newcommand{\Tbij}[4]{%
	\xymatrix@M=0pt{{#1\,}\ar@<2pt>@{->>}[r]^(.4){\scriptscriptstyle\smash{\interp{#2}}}&\,{#3}\ar@<1pt>@{->>}[l]^(.4){\scriptscriptstyle\phantom{x}\smash{\interp{#4}}\phantom{x}}}}
\newcommand{\build}[1][N]{\interp{Build}_{#1}}
\newtheorem*{conj}{Conjecture~\ref{conj:inverse}'}
\newcommand{\down}[1]{\mathord{\downarrow}#1}

\title{Decomposing graphs into stable and ordered parts}
\titlerunning{Decomposing graphs into stable and ordered parts}
\author{Hector Buffi\`ere}{Université Paris Cité, CNRS, IRIF, Paris, France \and Centre d'Analyse et de Mathématique Sociales, CNRS/EHESS, Paris, France}{buffiere@irif.fr}{https://orcid.org/0009-0004-1385-1813}{}
\author{Patrice Ossona de Mendez}{Centre d'Analyse et de Mathématique Sociales CNRS UMR 8557, France \and  Computer Science Institute of Charles University (IUUK), Praha, Czech Republic}{pom@ehess.fr}{https://orcid.org/0000-0003-0724-3729}{}
\authorrunning{H.\ Buffi\`ere, P.\ Ossona de Mendez} %
\Copyright{Hector Buffi\`ere, Patrice Ossona de Mendez} %



\category{} %

\relatedversion{} %

\supplement{}%

\funding{
\phantom{.}\hfill\phantom{.}\linebreak
\begin{minipage}{.6\textwidth}
This paper is part of a project that has received funding from the European Research Council (ERC) under the European Union's Horizon 2020 research and innovation programme (grant agreement No 810115 -- {\sc Dynasnet})
\end{minipage}
\hfill\begin{minipage}{.3\textwidth}
\includegraphics[width=\textwidth]{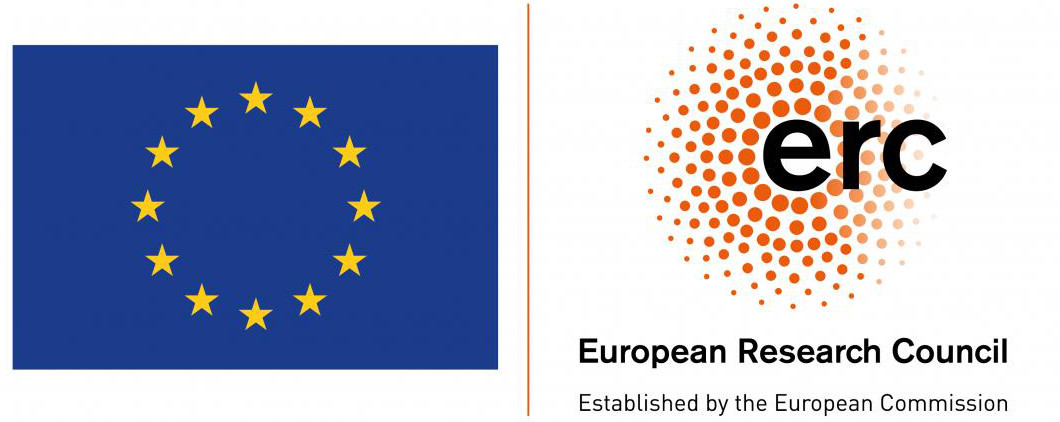}
\end{minipage}\hfill
}

\acknowledgements{}%

\hideLIPIcs  %

\EventEditors{John Q. Open and Joan R. Access}
\EventNoEds{2}
\EventLongTitle{42nd Conference on Very Important Topics (CVIT 2016)}
\EventShortTitle{CVIT 2016}
\EventAcronym{CVIT}
\EventYear{2016}
\EventDate{December 24--27, 2016}
\EventLocation{Little Whinging, United Kingdom}
\EventLogo{}
\SeriesVolume{42}
\ArticleNo{23}

\begin{document}
	\maketitle
	\begin{abstract}
    Connections between structural graph theory and finite model theory recently gained a lot of attention. In this setting, many interesting questions remain on the properties of dependent (NIP) hereditary classes of graphs, in particular related to first-order transductions.

    In this paper, we study modelizations (which are strong forms of transduction pairings) of classes of graphs by classes of structures. In particular, we consider models obtained by coupling a partial order and a colored graph (thus forming a partially ordered colored graph). 
    
     Motivated by Simon's decomposition theorem of dependent types into a stable part and a distal (order-like) part,
     we conjecture that every dependent hereditary class of graphs admits a modelization in a monadically dependent coupling of a class of posets with bounded treewidth cover graphs and a monadically stable class of colored graphs.
     
     In this paper, we consider the first non-trivial case (classes with bounded linear cliquewidth) and prove that the conjecture holds in a strong form, 
     the model class being a monadically dependent coupling of a class of  disjoint unions of chains and a class of colored graphs with bounded pathwidth. 
     
     We extend our study to classes that admit bounded-size bounded linear cliquewidth  decompositions and prove that they have a modelization in a monadically dependent coupling of a class of disjoint unions of chains and a class of colored graphs with bounded expansion, the model class  also admitting bounded-size bounded linear cliquewidth decompositions.
	\end{abstract}
\section{Introduction}
In recent years, there has been a growing interest in the structural study of hereditary
classes of graphs in the light of model theory. Particularly, since the identification in
\cite{Adler2013} of a bridge between the graph theoretical notion of a nowhere dense class \cite{ND_logic,ND_characterization}
and the model theoretical notions of stability \cite{shelah1969stable},  dependence \cite{Shelah2004}, and their monadic
counterparts~\cite{BS1985monadic}, a specific attention has been given to
stability and dependence in the context of structural and algorithmic graph theory. These studies are deeply motivated by the conjecture that First-Order ({\sf FO}) model checking is fixed-parameter tractable on a hereditary class of graphs
$\mathscr C$ if and only if the class $\mathscr C$ is dependent. 

These works recently led to structural and combinatorial characterizations of stability and dependence of hereditary classes of graphs \cite{covers,svm2024,dreier2024flipbreakability,Bonnet2025}.
Of particular interest, for example, is the characterization of  hereditary classes of graphs with dependent ordered expansions, which are exactly the classes of graphs with bounded twin-width \cite{Tww_ordered}.

Despite all these characterization theorems, the theory lacks a description of how dependent hereditary classes of graphs can be generated from simple dependent classes, and whether they can be put in a normal form.

For hereditary classes of graphs, the notions of stability and dependence are equivalent to their monadic forms \cite{braunfeld2022existential}. It follows that, in this context, stability and dependence can be defined using the (model-theoretic) notion of (first-order) transduction, a way to
logically interpret classes of structures from other colored structures (cf \zcref{sec:trans} for a formal definition). 
Precisely, a hereditary class of graphs $\mathscr C$ is \emph{dependent} (resp. \emph{stable}) if no transduction of $\mathscr C$ contains all finite graphs (resp. all half-graphs\footnote{The half-graph of size $k$ is the bipartite graph with vertices $a_1,\ldots,a_k,b_1,\ldots,b_k$ such that $a_i$ and $b_j$ share an edge if and only if $i\le j$.}).
In particular, if a hereditary class $\mathscr D$ is a transduction of a dependent (resp. stable) hereditary class $\mathscr C$, then so is $\mathscr D$.
Therefore, it seems natural to seek for normal forms based on the existence of a transduction from well-understood classes.
In this context, the next conjecture has been proposed to characterize stable hereditary classes of graphs. 

\begin{conjecture}[{\cite[Conjecture 6.1]{RW_SODA}}]
	\label{conj:ND}
	A hereditary class of graphs is stable if and only if it is a transduction of a nowhere dense class of graphs.
\end{conjecture}

On the other hand, within the realm of dependent hereditary classes of graphs, stable ones have a very simple combinatorial characterization:
A dependent hereditary class of graphs is stable if and only if the graphs in the class excludes some fixed half-graph as a semi-induced subgraph \cite{RW_SODA}. As such, one would want to decompose the graphs of a dependent class into a stable part and a part consisting of the gluing of half-graphs. While the gluing may be simple in the easier cases, one expects some branching structure for general dependent classes. Quoting Shelah in~\cite{shelah2015dependent}:``Our thesis is that the picture of dependent theory is the combination of the one for stable theories and the one for the theory of dense linear order or trees (\dots)''.

Instead of simply decomposing the graph structure itself, we are seeking  decompositions that are, in a sense, logically equivalent to the graphs they describe. 
That is why  we introduce the following concept of \emph{modelization}.
A modelization of a class of graphs $\mathscr C$ in a class of models $\mathscr M$ is given by an
interpretation $\interp{I}$ of $\mathscr C$ in $\mathscr M$ and, conversely, a modelization transduction $\interp{T}$ from $\mathscr C$ to $\mathscr M$, such that $\interp{I}\circ \interp{T} = \interp{Id}$ (with the specific requirement that the vertex set of every graph should be traceable in its models; see \zcref{def:modelization}). This notion is stronger than the one of \emph{transduction pairing}, which is based on encoding/decoding possibilities, and the one of \emph{transduction equivalence}, which is at class level.

As the structure of general dependent hereditary classes of graphs is far too complex, we initiate our study by focusing on the first non-trivial example, which is the case of classes with bounded linear cliquewidth, before extending it to classes admitting bounded-size bounded linear cliquewidth decompositions.
The logical properties of graphs of bounded linear cliquewidth has been an active field of
study in the last 30 years, with a recent interest in defining their decompositions by
transductions. It was first proven in \cite{courcelle1995logical} that classes
of bounded linear cliquewidth are {\sf MSO}-transductions of paths,
or equivalently \cite{Colcombet2007}, {\sf FO}-transductions of linear orders.
In the other direction, it was shown in \cite{bojanczyk_definable_2021} that
graphs with bounded linear cliquewidth could {\sf MSO}-transduce bounded
width clique decompositions of themselves, already making use of Simon's
factorization theorem \cite{simon1990factorization}. The problem of transducing tree-like decompositions using {\sf MSO} logic or extensions of
it was also considered in \cite{campbell_recognisability_2025, campbell2025role, campbell_cmso-transducing_2025}. When considering the
more restrictive {\sf FO} logic, the problem becomes more difficult, since
half-graphs replace paths as the basic unstable structure, and are tougher
to handle. A first step was achieved in the stable setting in \cite{SODA_msrw,msrw}, where it was proved that every stable class with bounded linear cliquewidth admits a {\sf FO}-transduction pairing with a class with bounded pathwidth. %

\subsection{Our results}

A central tool in our study are bounded-size decompositions on a given base class.
The notions of bounded-size and quasibounded-size decompositions with bounded treedepth base classes have been introduced by Ne{\v s}et{\v r}il and the second author years ago and are central to the structural theory of graph sparsity, providing a characterization of both classes with bounded expansions and nowhere dense classes  \cite{POMNI,ND_characterization} and, more recently, of monadically stable classes of graphs \cite{covers}. 
First, we make use of bounded size decompositions results proved in \cite{SODA_msrw,msrw} as an initial step toward the proof of our characterization result of classes with bounded linear cliquewidth.
This strong generalization of the results proved in \cite{SODA_msrw} shows that each class with bounded linear cliquewidth has a modelization in a class of structures composed of a very simple partial order (union of chains) and a stable edge relation (defining a graph with bounded pathwidth) on the same domain. (We call \emph{coupling} such classes of composites when they are monadically dependent.)

Toward this end, we first analyze our base classes (which we call \emph{sob}, for shallow ordered bicographs) and prove a decomposition result for them.

\begin{restatable*}{theorem}{ThmSobs}
\label{thm:sobs}
    Let $\mathscr C$ be a class of sobs of bounded height. Then 
    $\mathscr C$ 
    has a modelization in a class
    $\mathscr D$, which is a coupling of a class of trees
    of bounded height and a class of posets formed by
    unions of chains.
\end{restatable*}

Then, by using a gluing technique, we deduce the following decomposition and  characterization theorem for classes with bounded linear cliquewidth.

\begin{restatable*}{theorem}{thmMain}
	\label{th:main}
	Every class of graphs with bounded linear cliquewidth has a modelization in a coupling  of a class of colored graphs with bounded pathwidth  and a class of  disjoint unions of linear orders.
\end{restatable*}

Moreover, we prove that the poset and the sparse graph forming the decomposition can be merged into a single poset with sparse cover graph.

\begin{restatable*}{theorem}{lcworders}
	\label{lcw_pairing_poset}
	A class of graphs $\mathscr C$ has bounded linear cliquewidth if and only if it has
	a modelization in a class of colored posets whose cover graphs have bounded pathwidth.
\end{restatable*}

Actually, we conjecture that this characterization of classes with bounded linear cliquewidth  extends in a natural way to a characterization of classes with bounded cliquewidth.

\begin{restatable*}{conjecture}{conjcw}
\label{conj:cw}
    A class of graphs $\mathscr C$ has bounded cliquewidth if and only if it has
	a modelization in a class of colored posets whose cover graphs have bounded treewidth.
\end{restatable*}

Having proved these characterizations of classes with bounded linear cliquewidth, we consider the more general setting of classes admitting bounded-size bounded linear clique width decompositions. These classes are linearly $\chi$-bounded (as a consequence of the linear $\chi$-boundedness of classes with bounded linear cliquewidth \cite{SODA_msrw,msrw}) and 
encompass the classes with bounded linear cliquewidth, the transductions of classes with bounded expansion, as well as, for example, the class of proper circular arc graphs. 
Though it has been proved that these classes are monadically dependent \cite{covers}, it is 
not known whether the property of admitting bounded-size bounded linear cliquewidth decompositions is preserved by transductions, which makes their study  quite challenging.
Despite this, we were able to extend our decomposition  to all the classes admitting bounded-size bounded linear cliquewidth decompositions. This is the main result of our paper.

\begin{restatable*}{theorem}{thmStrong}
    \label{strong_thm}
	Assume a hereditary class $\mathscr C$ admits bounded-size bounded linear cliquewidth decompositions. Then, it has a modelization in a coupling $\mathscr D$ of a class of graphs with bounded expansion and a class of  disjoint unions of linear orders.
    
    Moreover, the class $\mathscr D$ also admits bounded-size bounded linear cliquewidth decompositions.
\end{restatable*}

Motivated by our results, we conjecture that this theorem can be broadly generalized to a general decomposition result for all dependent (unstable)
classes of graphs (See discussion in \zcref{sec:conc}).

\begin{restatable*}{conjecture}{conjMain}
	\label{conj:inverse}
    Every  monadically dependent class $\mathscr C$ of graphs has a modelization in a coupling of 
	  a monadically stable class of binary structures and a class of posets having  bounded treewidth cover graphs.
\end{restatable*}

\subsection{Overview of the paper}
This paper is structured as follows: 
\begin{itemize}
    \item 
\zcref{sec:prelim} reviews the preliminary notions used in the rest of the paper and introduce the notions of coupling and modelization.
\item 
In \zcref{sec:trees}, we introduce %
the notion of semi-plane rooted tree, on which our models for tame unstable graph
classes are built and give, as a direct application, a modelization of shallow cographs in shallow cotrees.
\item 
In \zcref{sec:sob}, we introduce
the notion of \emph{sob} (shallow ordered bicograph) and prove the core modelization result (\zcref{thm:sobs}).
\item 
This result is lifted in \zcref{sec:gluing} to classes 
with bounded linear cliquewidth (\zcref{th:main,lcw_pairing_poset}),
using bounded-size decompositions (on the base class of cographs and sobs).
\item 
\zcref{th:main} is further extended in \zcref{sec:lcw_dec} to classes admitting bounded-size bounded linear cliquewidth decomposition (\zcref{strong_thm}); the main part of the section is devoted to the proof that the constructed coupling also admits bounded-size bounded linear cliquewidth decompositions, drawing inspiration from \cite{SBE_TOCL}, where it is proved that transduction of classes with bounded expansion admit bounded-size bounded shrubdepth decompositions.
\item 
In \zcref{sec:conc}, we conclude with some remarks and conjectures.
\end{itemize}

\section{Preliminaries}\label{sec:prelim}

\subsection{Basic model theory} 
A {\em{relational signature}} $\sigma$ is a
finite set of
relation symbols each with an arity.
 A \emph{$\sigma$-structure} $\mathbf M$  consists of a finite 
universe $M$ and an interpretation of  the symbols from the
signature: each $k$-ary relation symbol $U\in \sigma$ is 
interpreted as a subset $U(\mathbf M)$ of $M^k$.  
If $\mathbf M$ is a relational structure and $X\subseteq M$,
then we define the \emph{substructure} 
$\mathbf M[X]$ of $\mathbf M$
\emph{induced by} $X$ in the usual way.
We may view unary relation symbols as defining subsets of the domain, and write $x\in A$ for
$A(x)$.

For a signature $\sigma$, we consider the standard first-order logic over $\sigma$, which we denote by $\mathsf{FO}[\sigma]$ (or $\mathsf{FO}$ when $\sigma$ is clear from the context). For a formula $\varphi\in\mathsf{FO}[\sigma]$ with $k$ free variables and a $\sigma$-structure $\mathbf M$, we define
\[
\varphi(\mathbf M)=\{(v_1,\dots,v_k)\in M^k\colon \mathbf M\models \varphi(v_1,\dots,v_k)\}.
\]

\subsection{Partial orders}

A \emph{strict partial order} is an  irreflexive, asymmetric, and transitive binary relation.
A \emph{partially ordered set} (or \emph{poset}) is a pair $(X,<)$
where $X$ is a set and $<$ is a strict partial order.
A \emph{linear order} is a partial order in which every two elements are comparable.

Let $(X,<)$ be a poset.
A subset $A$ of $X$ is an \emph{antichain} (resp. a \emph{chain}) if every two distinct elements are incomparable (resp. comparable).
The \emph{downset} \emph{generated} by $a\in X$ is
$\down a:=\{x\in X\colon x\preceq a\}$. By extension, 
if $A$ is a subset of $X$, then $\down{A}:=\bigcup_{a\in A}\down{a}$.

A pair $(a,b)$ of elements of a poset $P$ is a \emph{cover} of $P$ if $a<_P b$ and there exists no element $c$ with $a<_P c<_P b$. 
The cover relation is, as usual, denoted by appending a `:' to the partial order symbol 
($<:$ for $<$, or $\prec:$ for $\prec$).
Hence, $x<:y$ if $(x,y)$ is a cover.
The \emph{cover graph} of a poset $P$ is the graph with vertex set $P$, where two elements $a$ and $b$ are adjacent if  $(a,b)$ or $(b,a)$ is a cover of $P$.

A finite partial order $\prec$ is a \emph{tree-order} if 
	it has a (unique) minimum and for every element $v$, 
    the set $\down{v}$ is a  chain.
Note that a finite rooted tree $(T,r)$ defines a tree-order, which is the partial ordering on the vertices of $T$   with $u \preceq v$  if and only if the unique path from the root to $v$ passes through $u$.
Every tree-order is a \emph{meet-semilattice}, meaning that every two vertices $x,y$ of $T$ have a greatest lower bound $x\curlywedge y$, which is called the \emph{infimum} of $x$ and $y$.

\subsection{Graphs and colored graphs}  
Directed graphs can be viewed as finite structures over the 
signature consisting of a binary relation symbol~$E$, interpreted 
as the arc relation, in the usual way. If $E$ is interpreted by
a symmetric relation, which is then the edge (or adjacency) relation, then the structure represents
a (simple, undirected, and loopless) graph. 
When we speak of a graph, we mean an simple undirected loopless graph.
For the sake of clarity, we denote by $V(G)$ the domain of a graph $G$, that is its vertex set.

For a finite label set $\Lambda$, by a {\em{$\Lambda$-colored}}
graph we mean a graph enriched by a unary predicate~\(U_\lambda\)
for every $\lambda\in \Lambda$. In this paper, by \emph{partially ordered graph} we mean
a graph equipped with a partial order on its set of vertices.

We say that $H$ is a {\em subgraph} of a graph $G$ and write $H\subseteq G$ if $V(H)\subseteq V(G)$ and $E(H)\subseteq E(G)$.
We say that it is and {\em induced subgraph} and write $H\subseteq_i G$ or $H=G[V(H)]$ if $E(H) = E(G)\cap (V(H)\times V(H))$.
A bipartite graph $H$ with parts $A$ and $B$ is a {\em semi-induced subgraph} of a graph $G$ if $E(H) = E(G)\cap(A\times B\cup B\times A)$, and we shall write $H=G[A,B]$. The \emph{bipartite complement}
$\overline{G}^{\mbox{\tiny bip}}$ of a bipartite graph $G$ with parts $A$ and $B$ is obtained from $G$
by complementing the edge relation between $A$ and $B$.

A class is called \emph{weakly sparse} (or \emph{biclique-free}) if there is an integer $t$ such that
none of its graphs contain a $K_{t,t}$ as a subgraph.
A class of graphs $\mathscr C$ is {\em degenerate} if there is an integer $d$ such that 
every subgraph of a graph in $\mathscr C$ has
 minimum degree at most $d$.

 A \emph{cograph} (or complement-reducible graph)
 is a graph that can be obtained from single-vertex graphs by disjoint unions and complete joins.
 Hence, a cograph $G$ can be described by a rooted tree (called \emph{cotree}) with internal vertices labeled $U$ (union) and $J$ (join), whose leaves are the vertices of $G$, where two vertices of $G$ are adjacent if and only if their lowest common ancestor in the cotree is labeled by $J$. 
 Note that this representation is unique if we requite the labels $U$ and $J$ to alternate on every branch of the cotree.
 Also note that cographs are characterized by the property of excluding the path $P_4$
on $4$ vertices as an induced subgraph.

\subsection{Transductions} 
\label{sec:trans}
A \emph{transduction} maps a $\sigma$-structure 
to a set of $\sigma'$-structures. 
It is defined as a composition of 
the following types of atomic operations:
\begin{itemize}
	\item \textbf{Copying.} A \emph{$k$-copying} is an operation that
	given a structure~$\mathbf{M}$ outputs a disjoint union of $k$ copies of
	$\mathbf{M}$ extended with a symmetric binary relation which connects each vertex with its copies (which we call \emph{clones}).
    The atomic $k$-copying operation is noted $\interp{C}_k$.
	\item \textbf{Coloring.} A \emph{coloring} extends the 
	signature $\sigma$ by a new unary predicate 
	$U\not\in \sigma$. The output of a coloring applied 
	to a $\sigma$-structure $\mathbf M$ is the set of all 
	$\sigma\cup\{U\}$-structures whose restriction to 
	$\sigma$ is $\mathbf M$, that is, we interpret the unary
	predicate in all possible ways.  We sometimes speak of a
	\emph{monadic expansion} for a subclass of a coloring. 
	\item \textbf{Interpreting.} A \emph{(simple) interpretation} $\interp{I}$
	of $\sigma'$-structures in $\sigma$-structures is defined by 
	a formula $\rho_U\in\mathsf{FO}[\sigma]$ with $k$ free variables for each relation symbol $U\in\sigma'$ with arity $k$ and a formula $\nu(x)$ with a single free variable.
	Given a $\sigma$-structure $\mathbf{M}$, the $\sigma'$-structure
	$\interp{I}(\mathbf M)$ has universe $\nu(\mathbf M)$ and, for each $U\in\sigma'$, we define $U(\mathsf{I}(\mathbf M))=\rho_U(\mathbf M)$.
    The atomic interpretation expanding by a relation $U$ defined by $\rho$ is noted $\interp{Rel}^{\rho
\to U}$.
\end{itemize}

We take time for some comments about interpretations. First, it is useful to agree an interpretation does not modify a relation $U$ if no formula $\rho_U$ is associated to $U$. 
(In other words, if no formula is associated to some  $U\in\sigma\cap\sigma'$, then we implicitly define $\rho_U(\bar x):=U(\bar x)$.)
Because of this, an interpretation $\interp{I}$ of $\sigma'$-structures in $\sigma$-structures uniquely extends as an interpretation of 
$\tau'$-structures in $\tau$-structures, for 
$\tau\supseteq\sigma$ and  $\tau'\subseteq \sigma'\cup\tau$ (where formulas defining symbols in $\sigma'\setminus\tau'$ are ignored). In particular, when defining an interpretation, it will be sufficient to consider the signatures $\sigma$ and $\sigma'$, where $\sigma$ is the set of all the relations used in some formula of the interpretation.

\emph{Transductions} are defined as compositions of  atomic transductions. As usual, we denote by $\interp{J}\circ\interp{I}$ the composition of two 
transductions $\interp{I}$ and $\interp{J}$, that is the mapping defined by
$\interp{J}\circ\interp{I}(\mathbf M)= \bigcup\{\interp{J}(\mathbf N)\colon \mathbf N\in\interp{I}(\mathbf M)\}$.
A fundamental property of transductions it that they can be rewritten as the composition of the copying operation, a coloring operation, and an interpretation (in this order) \cite{SBE_TOCL}. Let $\mathsf T$ be a transduction from $\sigma$-structures to $\sigma'$-structures.
For every $\sigma$-structure $\mathbf M$, the domain of a structure $\mathbf N\in\interp{T}(G)$ is naturally partitioned into equivalence classes, each class corresponding to a set of clones of an element of $\mathbf M$. %

As for interpretations, a transduction from $\sigma$-structures to $\sigma'$-structures
uniquely extends to a transduction from $\tau$-structures to $\tau'$-structures whenever $\tau\supseteq \sigma$ and $\tau'\subseteq \sigma'\cup\tau$. However, because of the coloring, we also allow $\tau'$ containing further unary relations, which are (by default) defined by a coloring operation. 

As a particular case, let us mention the transduction  from $\sigma$-structures to $\sigma'$-structures where $\sigma'\setminus\sigma$ is a set of unary relations, which defines all the undefined unary relations by a coloring operation, then
forgets all the relations in  $\sigma\setminus\sigma'$. In the case where $\sigma'=\sigma$, we will speak about the \emph{identity transduction} and denote it by $\interp{Id}$. In the case where $\sigma'\subset\sigma$, we shall call this  transduction a \emph{reduct transduction} and the image of a $\sigma$-structure $\mathbf M$ by this transduction is called the \emph{$\sigma'$-reduct} of $\mathbf M$.
Of particular interest is the 
\emph{hereditary closure transduction} $\mathsf{Her}$, which maps a relational structure
$\mathbf M$ to the set of all its induced substructures.

We say that a class $\mathscr D$ is a \emph{transduction} of a class $\mathscr C$ (or that the class $\mathscr C$ \emph{transduces} the class $\mathscr D$) if $\mathscr D\subseteq\interp{T}(\mathscr C)$ for some transduction~$\interp T$.
Two classes $\mathscr C$ and $\mathscr D$ are \emph{transduction-equivalent} if each of the classes $\mathscr C$ and $\mathscr D$ is a transduction of the other. 
Stronger than a transduction-equivalence, a \emph{transduction pairing} of two classes~$\mathscr C$
and~$\mathscr D$ is a pair $(\mathsf D,\mathsf C)$ of  transductions,
such that 
\[
\forall\mathbf A\in\mathscr C\quad\exists \mathbf B\in\mathsf D(\mathbf A)\cap\mathscr D\quad \mathbf A\in\mathsf C(\mathbf B),\quad\text{and}\quad
\forall\mathbf B\in\mathscr D\quad\exists \mathbf A\in\mathsf C(\mathbf B)\cap\mathscr C\quad \mathbf B\in\mathsf D(\mathbf A).
\]

\medskip 

\subsection{Stability and dependence}

Dependence and stability are notions from model theory, respectively expressing the
non-definability of  finite sets and of finite linear
orders.
Baldwin and Shelah~\cite{BS1985monadic} introduced the notion of \emph{monadic dependence} (resp. of \emph{monadic stability}),
where one requires each monadic expansion to be dependent (resp. stable), and proved (using a different terminology) that they
can be characterized using transductions as follows.

\begin{theorem}[\cite{BS1985monadic}] A class of structures is
    \begin{itemize}
        \item monadically dependent if and only if it does not
    transduce the class of all graphs.

        \item monadically stable if and only if it does not transduce
        the class of all linear orders.
    \end{itemize}    
\end{theorem}

In the case of hereditary classes of structures, it was shown \cite{braunfeld2022existential}
that dependence and stability are equivalent to their monadic counterparts. As such,
we will use the adjective monadic only when needed, emphasizing that a class of structures
is not assumed to be hereditary.

\subsection{Width parameters}

We omit the original definitions of the classical parameters treewidth, pathwidth, cliquewidth and linear cliquewidth,
as we shall only need the following  model theoretical characterizations of the boundedness of these parameters.

\begin{theorem}[\cite{Gurski2000}]
	\label{fact:ws}
A class has bounded pathwidth if and only if it is weakly sparse and has bounded linear cliquewidth;
a class has bounded treewidth if and only if it is weakly sparse and has bounded  cliquewidth.
\end{theorem}

\begin{theorem}[\cite{SODA_msrw,msrw}]
	Every monadically stable class with bounded linear cliquewidth is  transduction-equivalent to a class with bounded pathwidth.
\end{theorem}

\begin{theorem}[\cite{RW_SODA}]
	Every monadically stable class with bounded cliquewidth is  transduction-equivalent to a class with bounded treewidth.
\end{theorem}

\begin{theorem}[\cite{courcelle1995logical}]
	For a class of graphs $\mathscr C$, the following are equivalent:
\begin{itemize}
	\item  $\mathscr C$ has bounded linear cliquewidth;
	\item $\mathscr C$ is an {\sf MSO}-transduction of the class of paths;
	\item $\mathscr C$ is an {\sf FO}-transduction of the class of linear orders.
\end{itemize}
\end{theorem}

\begin{theorem}[\cite{courcelle1995logical}]
	For a class of graphs $\mathscr C$, the following are equivalent:
\begin{itemize}
	\item  $\mathscr C$ has bounded cliquewidth;
	\item $\mathscr C$ is an {\sf MSO}-transduction of the class of trees;
	\item $\mathscr C$ is an {\sf FO}-transduction of the class of tree-orders.
\end{itemize}
\end{theorem}

\subsection{Decompositions and covers}\label{sec:dec}

A graph class $\mathscr C$ admits a \emph{bounded-size decomposition of parameter $p\in \mathbb N$ on a base class~$\mathscr D$} if there
is an integer $C_p$ such that for every $G\in \mathscr C$, there is a coloring $\gamma : V(G)\to [C_p]$ such
that for any set $I\subseteq [C_p]$ of at most $p$ colors, the graph induced by vertices with colors
in $I$ belongs to $\mathscr D$.
Let $\Pi$ be a class property (like ``bounded linear cliquewidth'').
We say that a graph admits \emph{bounded-size $\Pi$-decompositions} if it admits decompositions of parameter $p$ on a base class $\mathscr D\in\Pi$ for every
integer $p$ (we may simply talk about decompositions in this text as all decompositions considered will be bounded-size).

This property is equivalent to having bounded-size  $\Pi$-covers \cite{covers}. A class $\mathscr C$ has \emph{bounded-size $\Pi$-covers}
if for every integer $p$ there is a class $\mathscr D_p\in\Pi$ and an integer $m_p$ such that for every $G\in \mathscr C$ there exists subgraphs
$G_1, \ldots, G_{m_p} \subseteq G$ such that each $G_i$ is in $\mathscr D_p$ and for every $p$ vertices
$u_1, \ldots, u_p \in V(G)$, there is $i\in [m_p]$ such that $u_1, \ldots, u_p \in V(G_i)$.

\subsection{Coupling}
\label{sec:coupling}
A class $\mathscr A$ is a \emph{coupling} of two monadically dependent classes $\mathscr B$ and $\mathscr C$, which we can
assume have distinct signatures $\sigma_{\mathscr B}$ and $\sigma_{\mathscr C}$, if $\mathscr A$ is monadically dependent and every $A\in \mathscr A$ is a
$\sigma_{\mathscr B}\cup \sigma_{\mathscr C}$-structure whose $\sigma_{\mathscr B}$-reduct is in $\mathscr B$, and
whose $\sigma_{\mathscr C}$-reduct is in $\mathscr C$.
The notion of  modelization  is a strengthening of the notion of a transduction pairing.
\subsection{Modelization}
\label{sec:model}
Let $\mathscr C$ be a class of graphs, let $\sigma$ be a finite relational signature including the unary predicate $\ground$. The set $\ground(\mathbf M)$ is called the \emph{ground} of $\mathbf M$.

\begin{definition}
\label{def:modelization}
	A \emph{modelization } of $\mathscr C$ in $\sigma$-structures is a pair $(\interp{I},\interp{T})$, where
	\begin{itemize}
		\item $\interp{I}$ is an interpretation of graphs in $\sigma$-structures such that the vertex set of $\interp{I}(\mathbf M)$ is $\ground(\mathbf M)$;
		\item $\interp{T}$ is a transduction from graphs to $\sigma$-structures such that for every $G\in\mathscr C$, there exists $\mathbf M\in\interp{T}(G)$ such that $G=\interp{I}(\mathbf M)$.
	\end{itemize}
\end{definition}

Given a (fixed) modelization  of $\mathscr C$, 
a $\sigma$-structure $\mathbf M$ is  a \emph{model} of $G$ if  $\interp{I}(\mathbf M)=G$, which is \emph{accessible}  if $\mathbf M\in  \interp{T}(G)$.
Note that the ground of an accessible model of a graph $G$ is the vertex set of $V$. We denote by $\mathscr C^\text{mod}$ the class of all the accessible models of graphs in $\mathscr C$ and we say that $\mathscr C$ has a \emph{modelization in} the class $\mathscr C^{\text{mod}}$.
Remark  that $(\interp{T},\interp{I})$ is a transduction pairing of 
$\mathscr C$ and $\mathscr C^\text{mod}$.

A mapping $(\mathbf M,X)\mapsto \restr(\mathbf M,X)$, defined for $\mathbf M$ a model of a graph in $\mathscr C$ and $X$ a subset of the ground of $\mathbf M$,
is a \emph{model restriction} if
$\restr(\mathbf M, X)$ is an induced substructure of $\mathbf M$ such that  $\interp{I}(\restr(\mathbf M, X)=\interp{I}(\mathbf M)[X])$.
We do not require that the restriction of an accessible model of a graph is itself accessible.

\section{Semi-plane rooted trees}\label{sec:trees}%

Tree models are ubiquitous in structural decompositions of %
graph classes.
They are usually unordered, but %
introducing ordering may be useful 
when dealing with unstable classes.
However, such an ordering should be in general partial if 
it is required to be transducible from the modeled graph.
This justifies the introduction of semi-plane rooted trees.

A plane tree is a classical combinatorial object consisting of a tree given with its plane embedding. Equivalently, it can be seen as a rooted tree with a left-to-right ordering of the children of each node.
We define a \emph{semi-plane rooted tree}
as a rooted tree, whose internal nodes may or may not have their children linearly ordered (See \zcref{fig:pt}). The \emph{height} of a semi-plane rooted tree is the number of vertices in a 
longest simple path from the root to a leaf.

\begin{figure}[ht]
\begin{center}
	\includegraphics[width=.75\textwidth]{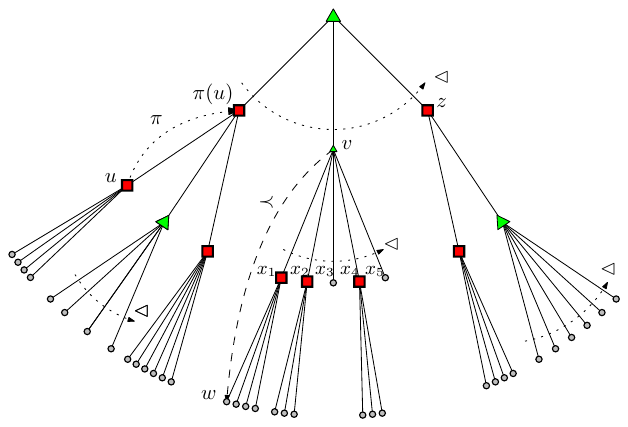}
\end{center}
\caption{A semi-plane rooted tree. The larger node is the root $r$. The parent function is denoted by $\pi$. 	The ancestor relation (which is a tree-order) is denoted by $\prec$. For example, $v\prec w$.
	Green (triangular) vertices are chain vertices (marked $C$), red (square) vertices are antichain vertices (marked $A$), and gray (circle) vertices are leaves (marked $L$).
	The children of a chain vertex are ordered by $\lin$. For example, the children of $v$ are ordered by $x_1\lin x_2\lin x_3\lin x_4\lin x_5$ and the children of the root are ordered by $\pi(u)\lin v\lin z$.}
\label{fig:pt}
\end{figure}

In this section,
we make precise their encoding as logical structures.
We define a  signature $\sigma_{1}=\{C,\curlywedge,\lin\}$, where $C$ is a  unary relation, $\curlywedge$ is a binary function, and $\lin$ is a binary relation.

The axioms of a semi-plane tree, encoded as $\sigma_1$ structure are:
\begin{itemize}
\item the axioms stating that $\curlywedge$ is the infimum function associated to a tree-order, namely  the standard axioms expressing that $\curlywedge$ defines a semilattice (idempotence, commutativity, associativity) and the special property 
$|\{x\curlywedge y,x\curlywedge z,y\curlywedge z\}|\leq 2$ for every $x,y,z$; note that for such a function, the partial order 
$\preceq$ defined by $(x\preceq y)\iff (x\curlywedge y=x)$ is indeed a tree-order (it has a minimum and, for every element $v$, the downset $\down{v}:=\{x\colon x\preceq v\}$ is a chain);
\item an axiom stating that the elements in $C$ are not maximal in the tree-order $\preceq$;
\item an axiom stating that $\vartriangleleft$ is a partial order in which two elements $x$ and $y$ are comparable if and only if 
their infimum is in $C$ and covers both $x$ and $y$.
\end{itemize}

It will be convenient to define additional relations and functions from those in $\sigma_1$. Towards this end,
we define the signature $\sigma=(C,\curlywedge,\vartriangleleft,<,E,L,A,r,\pi,{\rm st})$ and an interpretation $\interp{I}_{\sigma_1\rightarrow\sigma}$ of $\sigma$-structures in $\sigma_1$-structures (see \zcref{tab:sigma}).
\begin{table}[!ht]
\centering
\begin{tabular}{|c|Mc|p{.6\textwidth}|}
	\hlx{c{2-3}v[0]}
	\multicolumn{1}{c|}{}&\text{symbol}&\multicolumn{1}{c|}{meaning}\\
	\hlx{v[0]c{2-3}v[0-3]hv}
	\multirow{3}{*}{signature $\sigma_1$}&C(x)&$x$ is not a leaf, and its children are ordered\\
	\hlx{vc{2-3}v}
	&x\curlywedge y&the infimum of $x$ and $y$\\
	\hlx{vc{2-3}v}
	&x\vartriangleleft y&the sum, for each node $v$ marked $C$, of a linear order on the children of $v$\\
	\hlx{vhhv}
	\multirow{2}{*}{signature $\sigma_2$}&x<y&the transitive closure of the union of $\prec$ and $\vartriangleleft$\\
	\hlx{vc{2-3}v}
	&E(x,y)&the cover graph of $\prec$\\
	\hlx{vhhv}
	\multirow{6}{*}{extras: 
		$\sigma\setminus(\sigma_1\cup\sigma_2)$} &x\prec y& the tree-order associated to the infimum $\curlywedge$\\
	\hlx{vc{2-3}v}
	&L(x)&$x$ is a leaf (i.e. a maximal element) of the tree-order\\
	\hlx{vc{2-3}v}
	&A(x)&$x$ is not a leaf and the children of $x$ are unordered (hence, $C,L,A$ defines a partition of the node set)\\
	\hlx{vc{2-3}v}
	&M_r(x)&$x$ is the root of the tree (that is the minimum of $\prec$)\\
	\hlx{vc{2-3}v}
	&\pi(x)&function mapping $x$ to its unique predecessor in $\prec$ if $x$ is not the root, and to itself otherwise.\\
	\hlx{vc{2-3}v}
	&\mbox{st}_x(y)&function mapping $x$ and $y$ to the child $v$ of $x$ such that $v\preceq y$ if $x\prec y$, and to $x$ otherwise.\\
	\hlx{vh}
\end{tabular}
\caption{The signature $\sigma$ with intended meaning of the symbols.}
\label{tab:sigma}
\end{table}

We define the signature $\sigma_2=(<,E)$. 
We note that for every semi-plane tree $\Lambda$ represented as a $\sigma_1$-structure, the expanded $\sigma$-structure $\interp{I}_{\sigma_1\rightarrow\sigma}(\Lambda)$ can be recovered from its $\sigma_2$-reduct 
by an interpretation, which we denote by $\interp{I}_{\sigma_2\rightarrow\sigma}$.
Hence, the representations of semi-plane trees as $\sigma$-structures, $\sigma_1$-structures, and $\sigma_2$-structures are interpretation-equivalent.

This equivalence will be convenient as it will allow us to use all the power of $\sigma$-structures for semi-plane trees defined  as $\sigma_1$-structures or $\sigma_2$-structures.

Let us additionally define the subsignature $\sigma_3:=\{M_r,E,\vartriangleleft\}$ of 
$\sigma$.
While $\sigma_3$ is in general not as expressive as $\sigma$, it is the case when considering bounded-height trees.

\begin{remark}\label{rem:sigma3}
Let $\mathscr T$ be a class of bounded-height semi-plane trees. The trees of $\mathscr T$ 
viewed as $\sigma$-structures are interpretable from their
$\sigma_3$-reducts.
\end{remark}

\begin{proof}
Let $h$ be a bound on the height of semi-plane trees of $\mathscr T$. The order $\prec$
can be recovered in any tree of $\mathscr T$ by noticing that $x\prec y$ if and only if
there exists a simple $y$ -- $r$ path of length at most $h+1$ going through $x$. The infimum
$\curlywedge$ is then easily definable from $\prec$.
The children of a node are its larger (in $\prec$) $E$-neighbours, allowing to define $C(x)$
by the existence of two children $y_1$ and $y_2$ of $x$ with $y_1\vartriangleleft y_2$. Thus, $\sigma_1$-reducts are interpretable in the $\sigma_3$-reducts, and the
result follows.
\end{proof}

\subsection{T-models}

A \emph{T-model} of \emph{complexity} $(n,h)$ is a triple $\mathfrak M=(\mathbf
T,\gamma,\kappa)$, where $\mathbf T$  is a semi-plane tree with depth at most $h$, $\gamma:\ground(\mathbf T)\rightarrow [n]$ is a coloring of the ground of $\mathbf T$, and
$\kappa: A(\mathbf T)\cup C(\mathbf T)\mapsto \{\bot,\top\}^{[N]\times[n]}$ is a mapping such that
$\kappa(v)$ (written $\kappa_v$ in the following) is a symmetric function if $v\in A(\mathbf T)$.
(The mappings $\gamma$ and $\kappa$ are actually encoded using unary predicates, so that $\mathfrak M$ is a monadic expansion of $\mathbf T$).
The \emph{ground} of $\mathfrak T$ is defined as $L(\mathbf T)$.

We define the interpretation $\build[n]$ 
to graphs defined by
\begin{align*}
	\nu(x)&:=L(x)\\
	\rho_E(x,y)&:=\lambda(x,y)\vee\lambda(y,x),
	\intertext{where}
	\lambda(x,y)&=(A(x\curlywedge y)\land(x\neq y)\lor C(x\curlywedge y)\land (x<y))\land\kappa_{x\curlywedge y}(\gamma(x),\gamma(y))\\
	&=\neg(x\geqslant y)\land \kappa_{x\curlywedge y}(\gamma(x),\gamma(y))\\
\end{align*}

(Intuitively, the adjacency between a pair of vertices is stored at their least common ancestor in the tree,
and takes into account their order with respect to $<$ when this ancestor is a node of type $C$.)
This way,
a T-model $\mathfrak M$ defines a graph $G= \build[n](\mathfrak M)$.

For a T-model $\mathfrak M = (\mathbf T, \gamma, \kappa)$ and a node $u \in \mathbf T$, we
write $\mathfrak M{\upharpoonright}u$ for the submodel of $\mathfrak M$ induced by the subtree $\mathbf T[u]$ of $\mathbf T$ rooted at $u$ (and the derived restrictions of $\gamma$ and $\kappa$).

Let $\mathfrak M = (\mathbf T, \gamma, \kappa)$  be a T-model of a graph $G$  and let $X\subseteq\ground(\mathbf T)$. We denote by $\mathbf T\langle X\rangle$ the substructure of $\mathbf T$ induced by $\down{X}$ and by $\mathfrak M \langle X\rangle$
the substructure $(\mathbf T\langle X\rangle, \gamma|_X,\kappa|_{\mathbf T\langle X\rangle\setminus X})$ of $\mathfrak M$.
We emphasize the next remark, as we shall refer to it later.
\begin{remark}
	\label{rem:restr}
For every T-model $\mathfrak M$ of a graph $G$ and every subset $X$ of $V(G)$, 
  $\mathfrak M\langle X\rangle$ is a T-model of $G[X]$.
\end{remark}

In particular, the class of graphs with models of complexity $(n,h)$ is hereditary.

Our aim, in the next two sections, is to prove that for some important base classes of graphs $\mathscr C$, we can provide a transduction $\interp{T}$ such that $(\build[n], \interp{T})$ is a modelization  of $\mathscr C$.

In the case where $C(\mathbf{T})$ is empty, we retrieve the notion of connection model for classes of bounded shrubdepth (\cite{Ganian2012}). A class of graphs $\mathscr C$ has bounded shrubdepth
if and only if there exists integers $n$ and $h$ and a class $\mathscr M$ of  models of complexity $(n,h)$  without nodes of types $C$, such that $\mathscr C \subseteq \build[n](\mathscr M)$.  Otherwise, this corresponds to the notion of \emph{embedded shrubdepth} present in \cite{SODA_msrw}: a class of graphs $\mathscr C$ has bounded embedded shrubdepth
if and only if there exists integers $n$ and $h$ and class of models $\mathscr M$ each of complexity $(n,h)$ such that $\mathscr C \subseteq \build[n](\mathscr M)$.

\subsection{Shallow cographs}
The \emph{height} of a cograph $G$ is the minimum height of a cotree representing $G$.

A cotree of height $h$ is nothing but a T-model of complexity $(1,h)$ without C-nodes, where the function $\kappa$ is defined as
\[
\kappa_u(1,1)=\begin{cases}
    \bot&\text{ if $u$ is a disjoint-union node (type U),}\\
    \top&\text{ if $u$ is a join node (type J).}
\end{cases}
\]

Hence, by a cotree of height $h$ we mean a semi-plane rooted trees $\mathbf T$ of height $h$ without $C$-nodes, whose internal vertices are labeled U or J.

A cotree $\mathbf T$ is \emph{clean} if the types of the internal vertices alternate (between U and J) on every branch and every internal node has at least two children. 

It is easily checked that for every integer $h$ there exists an transduction $\interp{Cotree}_h$, such that if $G$ is a cograph with height at most $h$, then $\interp{Cotree}_h$ contains the unique clean cotree of $G$ (which has height at most $h$).
Hence, we have
\begin{remark}
\label{rem:cograph_model}
The pair $(\build[1],\interp{Cotree}_h)$ is a modelization  of the class of cographs with height at most $h$ in cotrees and that $(\mathbf T,X)\mapsto \mathbf T\langle X\rangle$ is a model restriction (by \zcref{rem:restr}). 

In this modelization, a cotree with height at most $h$ is clean if and only if it is accessible.    
\end{remark}

The  case of bipartite graphs admitting models of complexity $(2,h)$ is far more intricate, as we shall see in the next section dedicated to the proof of a statement similar to \zcref{rem:cograph_model}.

\section{Shallow ordered bicographs}\label{sec:sob}
  
A
\emph{semi-plane bicotree} (or simply, a \emph{bicotree}) 
is a T-model of complexity $(2,h)$ (with internal vertices labeled U,B, or O), where the function $\kappa$ is defined as

\[
\kappa_u(i,j)=\begin{cases}
    \bot&\text{if $u$ has type U (Union)};\\
    (i\neq j)&\text{if $u$ has type B (Bijoin)};\\
    (i<j)&\text{if $u$ has type O (Ordered-bijoin)}.
\end{cases}
\]
A \emph{shallow ordered bicographs} (or \emph{sob}) is a bipartite graph with parts $V_1$ and $V_2$ having a bicotree with $\gamma^{-1}(i)= V_i$ as a T-model.
The \emph{height} of a sob is defined as the minimum height of all its 
bicotrees. We will consider classes of sobs with bounded height.
(When the height is not bounded, these models appear in the literature using the name $K+S$ decomposition~\cite{fouquet1999bipartite}.)

We say that a bicotree $\mathfrak T$ is \emph{clean} if it satisfies the following conditions:

\begin{enumerate}
    \item For every U-node $u$ with children $v_1, \ldots, v_n$, the subtrees $\mathfrak T{\upharpoonright}v_i$ model the connected components of $\mathfrak{T}{\upharpoonright}u$;
    \item For every B-node $u$ with children $v_1, \ldots, v_n$, the subtrees $\mathfrak
    	T{\upharpoonright}v_i$ model the bipartite complement of the connected
		components of $\overline{\mathfrak{T}{\upharpoonright}u}^{\mbox{\tiny bip}}$;

    \item For every O-node $u$ with children $v_1, \ldots, v_n$, the subtrees $\mathfrak T{\upharpoonright}v_i$ contain both a leaf colored~1 and a leaf colored~2.
\end{enumerate}

For $\mathfrak T$ a bicotree, let $\mathfrak T^\star$ be the bicotree obtained by
exchanging labels U and B and reversing the order $\vartriangleleft$.

\begin{remark}
\label{rem:dual}
    If $\mathfrak T$ models $G$, $\mathfrak T^\star$ models $\overline{G}^{\mbox{\tiny bip}}$. Moreover, $\mathfrak T$ is clean if and only if $\mathfrak T^\star$ is clean.
\end{remark}

We now prove that every sob has a clean bicotree model. However,
as clean bicotrees can have internal nodes with a single child, a sob not only does not have a unique clean bicotree model, but has in fact clean bicotree models of unlimited heights.

\begin{restatable}{lemma}{lemcleanbic}
\label{lem:cleanbic}
    Every sob of height $h$ has a clean bicotree of height at most $3h$. 
\end{restatable}
\begin{proof}
	Let $G$ be a sob and let $\mathfrak T$ be a bicotree of $G$. 
    
    We prove by induction over $h\ge 1$ that for every bicotree $\mathfrak T$ of height at most $h$ there exists a clean bicotree $\mathfrak M$ of height at most $3h$ modeling  the same graph.
    
    All bicotrees of height $1$ are clean. Assume the property has been proved for $h-1$, where $h\ge 2$, and consider a bicotree $\mathfrak T$ of height $h$, which models a graph $G$.
	
	Assume the root $r$ is a node of type U. 
    Let $H_1,\dots,H_k$ be the connected components of $G$. As $r$ is of type U, each $V(H_i)$ is included in a subtree rooted at a child of $r$, hence has a 
    bicotree model of height at most $h-1$. By induction,
    there exists clean bicotrees $\mathfrak M_1,\dots,\mathfrak M_k$ of height at most $3(h-1)$ modeling  
    $H_1,\dots,H_k$. The bicotree with root labeled U and having $\mathfrak M_1,\dots,\mathfrak M_k$ as child subtrees, is a clean bicotree model of $G$ of height at most $3h$.
	
	Assume the root is a node of type B. Then, the same reasoning shows that $\overline{G}^{\mbox{\tiny bip}}$ has a clean bicotree model $\mathfrak M$ of height at most $3h$. Hence, according to \zcref{rem:dual}, $\mathfrak M^\star$ is a clean bicotree of $G$ of height at most $3h$.
	
	Assume the root $r$ is of type O and has children
    $v_1,\dots,v_n$.
    By induction, the graphs $G_1,\dots,G_k$ modeled by $\mathfrak T{\upharpoonright}v_1,\dots,\mathfrak T{\upharpoonright}v_k$ have clean bicotree models
    $\mathfrak M_1,\dots,\mathfrak M_k$ of height at most $3(h-1)$. 
    We partition $[k]$ into intervals 
    $I_L,I_1,\dots,I_\ell,I_R$ (in this order), where 
    \begin{itemize}
        \item for every $i\in I_L$, $V(G_i)$ is colored $1$,
        \item for every $i\in I_R$, $V(G_i)$ is colored $2$,
        \item for every $s\in [\ell]$ there is some $t_s\in I_s$ such that if $i\in I_s$, then  $V(G_i)$ has color 1 if $i<t_s$, color 2 if $i>t_s$ and $\bigcup_{i\in I_\ell}V(G_i)$ contains vertices of both colors.
    \end{itemize}
    The clean bicotree $\mathfrak M$ is constructed as follows (see \zcref{figsob}): the root $r$ has label B. If $I_L\neq\emptyset$, it has a child labeled U, whose children are the clean bicotrees $\mathfrak M_i$ for $i\in I_L$; if $\ell\geq 1$, $r$  has a child labeled $O$, with children $u_1,\dots,u_\ell$ (ordered this way) labeled $U$, each $u_s$ having children $\mathfrak M_i$ with $i\in I_s$;   if $I_R\neq\emptyset$, $r$ has a child labeled U, whose children are the clean bicotrees $\mathfrak M_i$ for $i\in I_R$. It is easily checked that $\mathfrak M$ is a clean bicotree model of $G$ of height at most $3h$.
    
\end{proof}
\begin{figure}[ht]
	\centering
	\includegraphics[width=.5\textwidth]{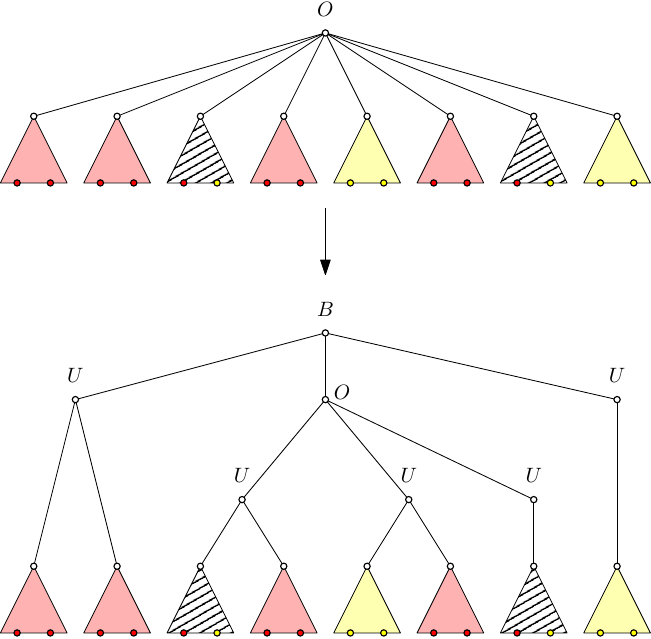}
    	\caption{Turning a bicotree into a clean bicotree, when the root is of type O. Color $1$ is red, color $2$ is yellow. Subtrees with monochromatic set of leaves are colored red or yellow; those with both colors are represented with a rising tiling pattern.}
	\label{figsob}
\end{figure}

The following structural property of sobs, which is similar to the property of cographs to exclude $P_4$ as an induced subgraph, implies that graph connectivity is first-order definable on sobs.

\begin{lemma}[\cite{fouquet1999bipartite}]
\label{sob_p7}
    Let $G$ be a sob, then $G$ does not contain an induced $P_7$.
\end{lemma}

\begin{theorem}
	\label{thm:bico_model}
	Let $\mathscr C$ be a class of sobs, 
    and let $h$ be an integer, such that every $G\in\mathscr C$ has a clean bicotree model of height at most $h$. %

    Then, we can define a transduction $\interp{Model}_h$, such that $(\build[2],\interp{Model}_h)$ is a modelization of $\mathscr C$, such that every clean bicotree models of height at most $h$ of a graph in $\mathscr C$ is accessible.

    Moreover, $(\mathbf T,X)\mapsto \mathbf T\langle X\rangle$ is a model restriction.
\end{theorem}
\begin{proof}
We first construct inductively a transduction $\interp{Model}_h$, such that $(\build[2],\interp{Model}_h)$ is a modelization of $\mathscr C$
and every clean bicotree models of height at most $h$ of graphs in $\mathscr C$ is accessible.

    As per \zcref{rem:sigma3}, we use the encoding of semi-plane trees as $\sigma_3$-structures.
    Let $\mathscr D$ be the class of all the clean bicotree models of height at most $h$ of every $G\in\mathscr C$.
    Let $\interp{Sk}$ be the natural interpretation of semi-plane rooted trees encoded as $\sigma_3$-structures in bicotrees. Note that if $\mathbf T=\interp{Sk}(\mathfrak T)$, then $\mathfrak T$ is a monadic expansion of an interpretation of $\mathbf T$ changing the representation of the semi-plane rooted tree. Hence, there is a transduction $\interp{X}_h$ such that 
    $(\interp{Sk},\interp{X}_h)$ is a transduction pairing of the classes $\mathscr D$ and $\mathscr D':=\interp{Sk}(\mathscr D)$.

    We construct a transduction 
    $\interp{T}_h$ such that 
    for every $\mathfrak T\in\mathscr D$, 
    $\interp{T}_h(\build[2](\mathfrak T))$
    contains $\interp{Sk}(\mathfrak T)$, hence such that 
    $\interp{X}_h\circ \interp{T}_h\circ\build[2](\mathfrak T)$ contains $\mathfrak T$.
    Then, as $\mathscr C=\build[2](\mathscr D)$, we deduce that
    $(\build[2],\interp{X}_h\circ\interp{T}_h)$ is a transduction pairing of $\mathscr C$ and $\mathscr D$. Thus, the transduction $\interp{Model}_h:=\interp{X}_h\circ\interp{T}_h$ fulfills the requirements of the theorem.
    
    We define the transduction $\interp{T}_h$ %
    by induction on $h$.

    If $h=1$, then $\mathscr C$ contains at most a $K_1$ colored $1$ and a $K_1$ colored $2$. These graphs have a unique clean bicotree model, namely themselves. 
    Hence, we can let $\interp{T}_1$ be the identity transduction with all the vertices marked $M_r$.

    Let $h>1$ and assume that $\interp{T}_{h-1}$ as been defined.
    The transduction $\interp{T}_h$ is defined using the next claim, which details how to construct (recursively) a tree model.
(This claim is illustrated by \zcref{fig:TrS}.)

\begin{claim}\label{eq_trans}
	Let $\sim\,\notin\sigma$ be a definable binary relation, which defines an equivalence relation on the domain of each sob in $\mathscr C$  and let $\mathsf T$ be a transduction from sobs  to bicotrees.
	
	Then, there exists a
	transduction $\interp{S}$ 
	such that for every non-empty sob $G\in\mathscr C$ with \mbox{$\sim$-equivalence} classes  $V_1, \ldots, V_k$ inducing substructures $G_i=G[V_i]$, $\interp{S}(G)$ contains a structure $H$ as follows (cf Fig.~\ref{fig:TrS}): 
	
	For all $H_1, \ldots, H_k$ where $H_i \in \interp{T}(G_i)$ and for all $a_i \in
	V(H_i)$, $H$ is obtained from the disjoint
	union of $H_1, \ldots, H_k$, and a new
	element $r$, by adding the adjacencies
	$E(a_i,r)$ (for $i\in[k]$).
\end{claim}
\begin{claimproof}
	The transduction $\mathsf S$ will be the composition
	\[\interp{S}=\interp{F}\circ\interp{S_0},\text{ where }\interp{S_0}:=\interp{P}\circ \interp{\Lambda}\circ \interp{C}\circ \interp{B}\circ \interp{Rel}^{\varphi_\sim
		\to \sim},\]
	where  $\varphi_\sim$ is the formula defining $\sim$	and the other transductions are defined  below.
	
	The interpretation $\interp{B}$ is defined by
	the formulas $\rho_U(\bar x):=U(\bar x)\wedge\bigwedge_{i\in[|\bar x|]}(x_1\sim x_i)$
	(for each relation $U$ in the signature $\sigma'$ of bicotrees).
	(Note that this interpretation breaks any relation containing elements of different equivalence classes.)
	
	Let us write $\interp T$ as the composition $\interp I\circ \interp \Lambda 
	\circ\interp C$, where $\interp I$ is an interpretation defined by the formulas
	$\nu(x)$ and $(\rho_{U})_{U\in \sigma'}$, $\interp \Lambda$
	is a monadic expansion, and $\interp C$ is a copying operation.%
    Without loss of generality, we can assume that $\interp \Lambda$ marks the first copy of $\interp{C}$ by a
    specific unary predicate $C_1$. %
    Define a new relation $\sim'$ using $C_1$ such that two points are equivalent under $\sim'$ if they have $\sim$-equivalent
    clones or are clones. %
	Let $\mu(x)$ and $\psi_U(\bar x)$ be the relativizations\footnote{the relativization $\varphi'$ of a formula $\varphi$ is recursively obtained 
		replacing $\exists y\ \varphi(\bar x,y)$ (resp. $\forall y\ \varphi(\bar x,y)$) by $\exists y.(y\sim
		x_1)\wedge\varphi'(\bar x,y)$ (resp. $\forall y.(y\sim x_1)\rightarrow\varphi'(\bar x,y)$)} of $\nu(x)$ and
	$\rho_{U}(\bar x)$ to $\sim'$-equivalence classes.
	The interpretation $\interp P$ is defined by $\mu(x)$ and the formulas
	$\psi_{U}(\bar x)\land\bigwedge_{i\in [|\bar x|]}(x_i\sim' x_1)$ for $U\in \sigma'$.

	For all non-empty $H_1, \ldots, H_k$ where $H_i \in \interp{T}(G_i)$, $\interp{S_0}(G)$ contains the  disjoint
	union of $H_1, \ldots, H_k$.

	The transduction $\interp{F}$ is the composition of the copy operation $\interp{C}_2$ and the transduction defined by $\nu(x)=M_1(x)\lor M_3(x)$,
	$\rho_{R'}(x,y):=M_2(x)\land M_2(y)\land R'(x,y)$, and
	$\rho_{L}(x,y):=M_2(x)\land M_3(y)\vee M_3(x)\land M_3(y)$, where $M_1,M_2,M_3$ are new unary predicates.
	For the particular choice of $M_1$ marking the first copy of the structure, $M_2$ marking exactly one element ($a_i$) in the first copy of each $H_i$, and $M_3$ marking a unique element ($a$) of the second copy of the structure, it is easily checked that the obtained structure has the required form.
	Hence, the transduction $\interp{S}$ fulfills all the requirements of the lemma.
\end{claimproof}

\begin{figure}[ht]
	\centering
	\includegraphics[width=\textwidth]{TrS2}
	\caption{The transduction $\sf S$, for $\sf T$ a copying operation.} 
	\label{fig:TrS}
\end{figure}

     Assume there exists a vertex in $G$ marked $M_{U,h}$ (intended meaning is that the root of the model is of type U), we define the equivalence relation 
        $x\sim y:=({\rm dist}(x,y)\le 6)$.
        According to \zcref{sob_p7}, the equivalence classes of $\sim$ are the connected components of $G$. According to \zcref{eq_trans}, there exists a transduction $\interp{T}_h$ such that if
        $G_1,\dots,G_k$ are the connected components of $G$ (that is, the subgraphs of $G$ induced by the equivalence classes of $\sim$),
        for every $\mathbf T_i\in\interp{T}_{h-1}(G_i)$, the graph obtained from the $\mathbf T_i$ by linking their root to a new vertex $r$ (then removing the marks $M_r$ everywhere and marking $M_r$ the root) is in 
        $\interp{T}_h$. Thus, by induction, $\interp{T}_h(G)$ contains every semi-plane rooted tree $\mathbf T$ such that $\interp{X}_h(\mathbf T)$ contains a model of $G$ with root marked U.
         
        Otherwise, the case where there exists a vertex in $G$ marked $M_{B,h}$ (intended meaning is that the root of the model is of type B) is handled in a similar way, by considering the bipartite complement of $G$.
    
        Otherwise (meaning is that the root of the model is of type O), 
    let $G_1, \ldots, G_n$ be the subgraphs modeled by the subtrees rooted
	at children of $u$. We first apply a monadic expansion $\interp K$ introducing the colors
    $\gamma(x) \in \{1,2\}$ of every vertex and predicates
	$K_1, K_2, K_3$, which we refer to as \emph{indices}, such that vertices of
	$G_i$ satisfy the predicate $K_{1+(i\mod 3)}$. We may write $\varkappa(x) = i$ as
	syntactic sugar for $\interp K_i(x)$.

    Let $x\in G_i$ and $y\in G_j$.
    If $\gamma(x)=1$, $\gamma(y)=2$, and $\varkappa(x)\neq \varkappa(y)$, then  $i<j$ if $x$ and $y$ are adjacent (i.e. $E(x,y)$) and $i>j$, otherwise.

    As every $G_i$ contains both a vertex colored $1$ and a vertex colored $2$, we can use the above property to check whether two vertices belong to the same $G_i$: we define the relation $\sim$ by
\begin{multline*}
x\sim y:=(\varkappa(x)=\varkappa(y))\wedge\\\Bigl[
(\gamma(x)=\gamma(y))\wedge\neg\chi_1(x,y)\wedge\neg\chi_1(y,x)
\vee(\gamma(x)\neq\gamma(y))\wedge\neg\chi_2(x,y)\wedge\neg\chi_2(y,x)\Bigr],
\end{multline*}
where
\begin{align*}
    \chi_1(x,y)&:=\exists z\  (\varkappa(z)\neq\varkappa(x))\land
		E(x,z)\land\neg E(z,y)\\
    \chi_2(x,y)&:=\exists z_1, z_2\ (\varkappa(x) \neq\varkappa(z_1))\wedge(\varkappa(x) \neq\varkappa(z_2)) \land
		E(x,z_1) \land \neg E(z_1,z_2) \land
		E(z_2, y)
\end{align*}

Under the assumption that $\gamma(x)=\gamma(y)$ and $\varkappa(x)=\varkappa(y)$, the formula $\chi_1$ asserts that there exists a vertex $z$ whose subtree is strictly to the right of the one of $x$ and strictly to the left of the one of $y$. Similarly, under the assumption $\gamma(x)\neq \gamma(y)$, $\chi_2$ expresses that there exist
vertices $z_1$ and $z_2$ which are in a subtree strictly to the right of the one of $x$ and in a subtree strictly to the left
of the one of $y$, respectively, and that $z_2$ is not left of $z_1$. It is easy to check that these two formulas cover all the possible cases.
Using these definitions of $\chi_1$ and $\chi_2$,
the transduction $\interp{T}_h$ is constructed as in the proof of \zcref{eq_trans}, 
except that the  interpretation $\mathsf P$ that appears in the construction also defines
$a_i\vartriangleleft a_j$ by the formula 
\begin{multline*}
\exists x,y\ (x\sim' a_i)\wedge(x\sim' a_j)\wedge (\gamma(x)=1)\wedge(\gamma(y)=2)\wedge\\
\Bigl[
\bigl((\varkappa(a_i)\neq\varkappa(a_j))\wedge E(x,y)\bigr)\vee
\bigl((\varkappa(a_i)=\varkappa(a_j))\wedge \chi_2(x,y)\bigr)
\Bigr].
\end{multline*}
This finishes the recursive construction of $\interp{Model}_h=\interp{X}_h\circ\interp{T_h}$. That
this transduction has the required properties is immediate from the construction.

\medskip

Then, that  $(\mathbf T,X)\mapsto \mathbf T\langle X\rangle$ is a model restriction follows from \zcref{rem:restr}.
\end{proof}

\ThmSobs

\begin{proof}
	By \zcref{thm:bico_model}, $\mathscr C$ admits a transduction
	pairing with the class of its clean bicotrees. It is then
	sufficient to remark that given that we have a bound on their
	height, the bicotrees can be represented as semi-plane trees
	using the signature $\sigma_3$ by \zcref{rem:sigma3}.
	
	It is then clear that the posets induced by $\vartriangleleft$
	are unions of chains.
\end{proof}

\section{Modelization of graphs with bounded splits}
\label{sec:gluing}

In this section, we use splits to leverage the transduction
pairing obtained in the previous section to more complex classes of graphs.

The main tool we shall use is the notion of a split.

\subsection{Splits}

Cographs and sobs with bounded height will serve as basic decomposition blocks for more complex graphs. This motivates the following definition.

\begin{definition}
	A \emph{split} of a graph $G$ of size $N$ and height $h$ is a coloring $\gamma:V(G)\rightarrow [N]$ such that
	\begin{itemize}
		\item for every $1\le i\le N$, $G[\gamma^{-1}(i)]$ is a cograph with height at most $h$;
		\item for every $1\leq i<j\le N$, $G[\gamma^{-1}(i),\gamma^{-1}(j)]$ is a sob with height at most $h$.
	\end{itemize}
\end{definition}

Bounded-size splits appear in every class with bounded-size bounded linear cliquewidth decompositions, as we prove now.

\begin{lemma}
	\label{lem:dec2split}
	Let $\mathscr C$ be a class of graphs. 
	If $\mathscr C$  admits a parameter $2$ bounded-size bounded linear cliquewidth decompositions, then 
	there exists integers $N,h$ such that every graph $G\in\mathscr C$ has a split of size $N$ and height $h$.
\end{lemma}
\begin{proof}
	We first prove the result for graphs having T-models of complexity $(N,h)$ for some fixed integers $N$ and $h$.
	\begin{claim}
		\label{osd_dec}
		Assume 	 $G$ has a model $\mathfrak M$ of complexity $(N,h)$.
		
		Then, $G$ has a split of size $N$ and height $h$.
	\end{claim}
	\begin{claimproof}
		
		As noted in \zcref{rem:restr},  $\mathfrak M\langle V_i\rangle$ is a T-model of $G[V_i]$. Since
		this model uses only one color, it is a cotree.
		Similarly, $\mathfrak M\langle V_i\cup V_j\rangle$ is a T-model of $G[V_i\cup V_j]$, which
		uses only two colors, which we may relabel 1 and 2.
		By modifying the function $\kappa$ on internal nodes by letting $\kappa(1,1)=\kappa(2,2)=\bot$, we get a T-model of $G[V_i,V_j]$.
		The only nodes that are not readily of one of the three allowed
		types are those vertex $u$ such that $\kappa_u(2,1) = \top$ and $\kappa_u(1,2)=\bot$. For those,
		we can simply consider the model obtained by reversing the order $\lin$ on children
		of $u$ and making $u$ a node of type O, which models the same graph. Hence, repeating
		this operation for every problematic node, we obtain a bicotree.
		
		Both the cotree and bicotree are obtained by removing nodes from the tree of $\mathfrak
		M$, hence their heights are at most $h$.
	\end{claimproof}
	
	Taking a product coloring of the parameter $2$ bounded-size bounded linear cliquewidth decomposition and of a
	parameter $2$ bounded-size bounded embedded
	shrubdepth decomposition of the base class with bounded linear cliquewidth (whose existence is proved in  \cite[Theorem 1.2]{SODA_msrw}), we get that  $\mathscr C$ has a parameter $2$ bounded embedded
	shrubdepth decomposition. Taking again a product coloring using \zcref{osd_dec}, $\mathscr C$ admits splits of bounded size and height.
\end{proof}
Note that the converse is not true: the class of all $2$-subdivided graphs admits splits of size $2$ and height $2$, but does not admit parameter $2$ bounded-size bounded linear cliquewidth decompositions.

\subsection{Amalgams of cotrees and bicotrees}
Let $V$ be a set with a partition $V_1,\dots,V_n$.
For $1\le i\le n$, let  $\mathbf T_i$ be a cotree  with ground $V_i$ and, for $1\le i< j\le n$, let 
$\mathbf T_{i,j}$ be a bicotree with ground $V_i\cup V_j$ and color classes (with respect to
$\gamma$) $V_i$ and $V_j$.
(We assume that the cotrees and bicotrees are encoded as $\sigma_3$-structures.)
We define the \emph{amalgam}
$\mathbf M:=\am{V,(\mathbf T_i),(\mathbf T_{i,j})}$ as the structure (see \zcref{fig:GC}) obtained from the disjoint union of $V$, the $\mathbf T_i$'s and the $\mathbf T_{i,j}$'s by making each $v\in V_i$ adjacent to the corresponding vertex in $\ground(\mathbf T_i)$ and in all of the $\ground(\mathbf T_{i,j})$ (for $j\neq i$), defining the new predicates $L_1,\dots,L_n$ by $L_i(\mathbf M)=V_i$, and by redefining
the predicate $\ground$ so that $\ground(\mathbf M)=V$.
Formally speaking, we used some copies of the set and structures involved. Thus, it will be useful to define the \emph{injection system} of  $\mathbf M$ as the tuple $(\iota,(\iota_i),(\iota_{i,j}))$ of canonical injections, where 
$\iota$ maps $V$ to $\ground(\mathbf M)$, $\iota_i$ maps $V_i$ to $\ground(\mathbf T_i)$, and $\iota_{i,j}$ maps $V_i\cup V_j$ to $\ground(\mathbf T_{i,j})$.
The \emph{height} of the amalgam $\mathbf M$ is  the maximum height of the $\mathbf T_i$'s and the $\mathbf T_{i,j}$'s.
We say that $\mathbf M$ is \emph{clean} if all the $\mathbf T_i$'s and all the $\mathbf T_{i,j}$'s are clean.
Moreover, for a subset $W$ of $V$ we define
\[
\am{V,(\mathbf T_i),(\mathbf T_{i,j})}\langle W\rangle:=\am{W,(\mathbf T_i\langle W\rangle),(\mathbf T_{i,j}\langle W\rangle)}.
\]

\begin{figure}[ht]
	\begin{center}
		\includegraphics[width=.75\textwidth]{amalg}
	\end{center}
	\caption{The amalgam $\am{V,(\mathbf T_i),(\mathbf T_{i,j})}$. The green parts are the grounds of the structures $\mathbf T_i$ and $\mathbf T_{i,j}$. The ground of the amalgam is the set $V$ (on the dashed circle).}
	\label{fig:GC}
\end{figure}

Note that $\am{V,(\mathbf T_i),(\mathbf T_{i,j})}\langle W\rangle$ is an induced substructure of $\am{V,(\mathbf T_i),(\mathbf T_{i,j})}$.

We define the interpretation $\interp{SBuild}_{n,h}$
of graphs in amalgams as follows: the vertex set
of $\interp{SBuild}_{n,h}(\mathbf M)$ is $\ground(\mathbf M)$.
For any two vertices $x,y\in\ground(\mathbf M)$ we define $\mathbf T(x,y)$ as the connected component of $\mathbf M\setminus \ground(\mathbf M)$ containing a neighbor $x'$ of $x$ in $\mathbf M$ and neighbor $y'$ of $y$ in $\mathbf M$. 
(Note that the connected components of $\mathbf M\setminus \ground(\mathbf M)$ have diameter at most $2h$, hence are first-order definable.)
Intuitively, $\mathbf T(x,y)$ selects the tree among the $T_i$'s and the $T_{i,j}$'s that can be used to decide whether $x$ and $y$ are adjacent.

Applying the interpretation $\build[1]$ or $\build[2]$ on $\mathbf T(x,y)$ (depending on whether $x$ and $y$ belong to the same set $L_i(\mathbf M)$ or not), we define $\rho_E(x,y)$ as the adjacency of $x'$ and $y'$.

\begin{lemma}
	\label{lem:Sbuild}
	Let $\gamma$ be a split of size $n$ and height at most $h$ of a graph $G$.
	For $1\le i\le n$,  let $\mathbf T_i$ be a (cotree) T-model of $G[\gamma^{-1}(i)]$ with height at most $h$ and, for $1\le i<j\le n$,  let $\mathbf T_{i,j}$ be a (bicotree) T-model of $G[\gamma^{-1}(i),\gamma^{-1}(j)]$ with height at most $h$.
	
	Then, $\interp{SBuild}_{n,h}(\am{V(G),(\mathbf T_i),(\mathbf T_{i,j})})=G$.
\end{lemma}
\begin{proof}
	The property is immediate from the definition of the interpretation $\interp{SBuild}_{n,h}$.
\end{proof}

We now prove that, conversely, in the case where all the $\mathbf T_i$ and $\mathbf T_{i,j}$ are clean, then $\am{V(G),(\mathbf T_i),(\mathbf T_{i,j})}$ can be obtained as a transduction of $G$.

\begin{lemma}
	\label{lem:amalg}
	Let $\mathscr C$ be a class of graphs admitting splits of size $n$ and height $h$.
    
    There is a transduction $\sparse_{n,h}$ such that 
    if $G$ has a split $\gamma$ of size $n$ and height $h$, $\mathbf T_i$ is a clean model of $G[\gamma^{-1}(i)]$, and  $\mathbf T_{i,j}$ is a clean model of $G[\gamma^{-1}(i),\gamma^{-1}(j)]$, then 
    $\am{V(G),(\mathbf T_i),(\mathbf T_{i,j})}\in \sparse_{n,h}(G)$.
\end{lemma}
\begin{proof}
		The transduction $\interp{Cotree}_h$ is the composition of a copy operation $\mathsf C_{p}$, a monadic expansion $\Gamma_1$, and an interpretation $\mathsf I_1$ defined by formulas  $\nu^1$ and $\rho_R^1$; the transduction $\interp{Bicotree}_h$ is the composition of a copy operation $\mathsf C_{q}$, a monadic expansion $\Gamma_2$, and an interpretation $\mathsf I_2$ defined by formulas  $\nu^2$ and $\rho_R^2$. We can assume that the set  $\mathcal \tau_1$ of unary predicates involved in $\Gamma_1$ and the set $\tau_2$ of unary predicates involved in  $\Gamma_2$ are distinct.
		
	Let $k=1+np+\binom{n}{2}q$.
	Let $\Lambda$ be a monadic expansion, defining  the predicates in \[
	\tau_1\cup\tau_2\cup\{L_i: 1\le i\le n\}\cup\{D_i\colon 1\le i\le n\}\cup \{D_{i,j}\colon 1\le i<j\le n\}.
	\]
	
For formulas $\varphi(\bar x)$	and $\psi(y)$, denote by $\rel{\varphi(\bar x)}{\psi}$ the $\psi$-relativization of $\varphi$, that is the formula 
recursively defined by 
\begin{align*}
\rel{\exists y\ \varphi(\bar x,y)}{\psi}&:=\exists y\ \psi(y)\wedge \rel{\exists y\ \varphi(\bar x,y)}{\psi}\\
\rel{\forall y\ \varphi(\bar x,y)}{\psi}&:=\forall y\ \psi(y)\rightarrow \rel{\varphi(\bar x,y)}{\psi}
\intertext{and, for atomic $R$,}
\rel{R(\bar x)}{\psi}&:=R(\bar x)\wedge\bigwedge_{i\in[|\bar x|]}\psi(x_i).
\end{align*}

	The interpretation $\mathsf I$ is defined by 
	\begin{align*}
\nu(x)&:=\bigvee_i\rel{\nu^1(x)}{D_i}\vee\bigvee_{i<j}\rel{\nu^2(x)}{D_{i,j}}\vee \bigvee_i L_i(x), \\
\rho_L(x)&:=\bigvee_i L_i(x),\\
\rho_E(x,y)&:=\bigvee_i\rel{\rho_E^1(x,y)}{D_i}\vee\bigvee_{i<j}\rel{\rho_E^2(x)}{D_{i,j}}\\
&\qquad \qquad\qquad
 \vee\bigl(E(x,y)\wedge\bigvee_i (L(x)\cap D_i(x)\cap L_i(y)\vee L(y)\cap D_i(y)\cap L_i(x))\bigr)
\intertext{and, for $R\notin \{L,E\}$, }
\rho_R(\bar x)&=\bigvee_i\rel{\nu^1(\bar x)}{D_i}\vee\bigvee_{i<j}\rel{\nu^2(\bar x)}{D_{i,j}}.
\end{align*}

We claim that the transduction $\sparse_{n,h}:=\interp{I}\circ\Gamma\circ\interp{C}_k$ has the required properties:
	There is a monadic expansion such that
the vertex $v$ of the first copy of $G$ is marked $L_i$ if $\gamma(v)=i$, the vertex sets of the next 
$np$ copies are marked $D_1,\dots,D_n$ by groups of $p$, the remaining copies are marked $D_{1,2},\dots,D_{n-1,n}$ by groups of $q$,  the copies marked $D_i$ are  marked by predicates in $\tau_1$ to be interpreted as $\mathbf T_i$ by $\interp{I}_1$ (this is possible as every clean cotree model of height at most $h$ is accessible by $\interp{Cotree}_h$), and  copies marked $D_{i,j}$ are marked by predicates in $\tau_2$ to be interpreted as $\mathbf T_{i,j}$
by $\interp{I}_2$ (this is possible, as every clean bicotree model of height at most $h$ is accessible by $\interp{Bicotree}_h$  by \zcref{thm:bico_model}). With this particular monadic expansion,  the structure interpreted by 
$\interp{I}$ is $\am{V(G),(\mathbf T_i),(\mathbf T_{i,j}}$.
	\end{proof}

	\begin{theorem}
		\label{thm:amalg_model}
		Let $\mathscr C$ be a class of graphs with splits of order $n$ and height  $h$.
		The pair 
		$(\interp{SBuild}_{n,h},\sparse_{n,h})$ is a modelization  of $\mathscr C$ in amalgams with height at most $h$ such that every clean amalgam model of a graph in $\mathscr C$ is accessible.
		
		Moreover, $(\mathbf T,X)\mapsto \mathbf T\langle X\rangle$ is a model restriction. 
	\end{theorem}
	\begin{proof}
		The result is a direct consequence of \zcref{lem:Sbuild,lem:amalg}.
		That $(\mathbf T,X)\mapsto \mathbf T\langle X\rangle$ is a model restriction
		follows from the fact that $\am{V,(\mathbf T_i),(\mathbf T_{i,j})}\langle W\rangle$  is a substructure of $\am{V,(\mathbf T_i),(\mathbf T_{i,j})}$ and the (defining)
		equality of $\am{V,(\mathbf T_i),(\mathbf T_{i,j})}\langle W\rangle$ and $\am{W,(\mathbf T_i\langle W\rangle),(\mathbf T_{i,j}\langle W\rangle)}$.
\end{proof}

\thmMain
\begin{proof}
According to \zcref{lem:dec2split}, there exist integers $n$ and $h$ such that $\mathscr C$ admits splits of size $n$ and height $h$.
According to \zcref{thm:amalg_model}, $(\interp{SBuild}_{n,h},\sparse_{n,h})$ is a transduction pairing of $\mathscr C$ and the class $\mathscr C^\text{mod}$ of its accessible models.
By construction, the $\vartriangleleft$-reduct of 
$\mathbf M\in \mathscr C^\text{mod}$ is a disjoint union of linear orders. On the other hand, the $E$-reduct of $\mathscr C^\text{mod}$ is a degenerate class (as obtained by taking the union of
a bounded number of trees and stars) that is a transduction of a class with bounded linear cliquewidth. It follows that the $E$-reduct of $\mathscr C^\text{mod}$ has bounded pathwidth.
\end{proof}

In order to derive the alternate 
characterization of classes of graphs with
bounded linear cliquewidth given in \zcref{lcw_pairing_poset},
we show how we can in general translate pairings with couplings into
pairings with posets only.

\begin{restatable}{lemma}{lemorderall}
	\label{lem:order_all}
    Let $\mathscr C$ be a class of graphs with a modelization in a class $\mathscr D$, which is a coupling of a class $\mathscr P$ of posets and a class $\mathscr G$ of colored graphs.
	
	Then, there exists a class $\mathscr P'$ of colored posets (with unary predicates $\ground,P_2,P_3,P_4$), such that 
    $\mathscr C$ has a modelization in $\mathscr P'$.
    
	Moreover, if both $\mathscr G$ and the class of the cover graphs of the posets in $\mathscr P$ are weakly sparse, we can require that the class of the cover graphs of the posets in $\mathscr P'$ is also weakly sparse.
\end{restatable}
\begin{proof}
	Let $\{<\}$, $\{E\}$, and
	$\varsigma=\{<,E\}$ be the signatures of the structures in $\mathscr P$ and $\mathscr P'$, in $\mathscr G$, and in $\mathscr D$, respectively.
	
	We first define the transduction $\interp{T}$, which is the composition of a $\interp{C}_4$ copying operation introducing the binary relation $F$ between clones, a coloring introducing
    predicates $P_1, P_2, P_3, P_4$, and the interpretation of the ground $\ground$ and poset $<$ by the following formula:
	\begin{align*}
        \rho_{\ground}(u)&:=P_1(u)\wedge \ground(u)\\
		\rho_<(u,v)&:=\biggl(P_4(u)\wedge P_4(v)\wedge (u<v)\biggr)\vee
		\biggl(F(u,v)\land P_1(v)\land \bigvee_{i\in \{2,3,4\}}P_i(u)\biggr)\\
		&\qquad\vee \biggl(P_2(u)\wedge (P_3(v)\vee P_1(v))\wedge (\exists v'\ P_2(v')\wedge F(v,v')\wedge E(u,v'))\biggr)
	\end{align*}
	The class $\mathscr P'\subseteq \interp{T}(\mathscr D)$ is then defined as the image of $\interp T$ where
    the predicate $P_i$ marks the $i$-th copies produced by $\interp C_4$. The construction is illustrated in \zcref{fig:order_all}.
	The meaning of $\rho_<$ is as follows:
	we keep the poset part of $\mathbf M$ between the $P_4$-vertices. The $P_1$ vertices will form an antichain of maximum elements, each element being greater that its clones. Then, we add a cover between the clone in $P_2$ of a vertex $u$ and the clone in $P_3$ of a vertex $v$ if $u$ and $v$ are adjacent in $\mathbf M$ (by relation $E$). Adding this cover, we have to also add the order relation between the clone of $u$ in $P_2$ and the clone of $v$ in $P_1$.
	
	Assume that both $\mathscr G$ and the class of the cover graphs of the posets in $\mathscr P$ are weakly sparse. 
	By an easy application of the pigeon-hole principle, if the cover graph of a poset in $\mathscr P'$ contains a large balanced complete bipartite graph, we can assume that this complete bipartite graph is between elements marked $P_i$ and elements marked $P_j$.
	This clearly would contradict either the assumption on $\mathscr P$ or the assumption on $\mathscr G$.
	
	Conversely, we define a interpretation $\interp{I}$ by
	\begin{align*}
		\nu(v)&:=\ground(v)\\
		\rho_<(u,v)&:=\exists u',v'\ P_4(u')\wedge P_4(v')\wedge (u'<u)\wedge (v'<v)\wedge (u'<v')\\
		\rho_E(u,v)&:=\exists u',v'\ P_2(u')\wedge P_3(v')\wedge (u'<u)\wedge (v'<v)\wedge (u'<v')
	\end{align*}
Let $(\interp{I}',\interp{T'})$ be a modelization of $\mathscr C$ in $\mathscr D$, 
	it is clear that $(\interp{I}'\circ\interp{I},\interp{T}\circ\interp{T}')$ is a modelization of $\mathscr C$  in $\mathscr P'$.
\end{proof}

\begin{figure}
    \centering
    \includegraphics[width=\textwidth]{order_all}
    \caption{The coupling of a poset and a graph (on the left, with cover graph in black and graph edges in red) is transduced into a poset (on the right). The adjacency of $u$ and $v$ on the left is reflected by the covers $u_2\prec v_3$ and $v_2\prec u_3$.}
    \label{fig:order_all}
\end{figure}

\lcworders
\begin{proof}
	If $\mathscr C$ is transduction-equivalent to a class $\mathscr P$ of posets whose cover graphs have bounded pathwidth, then $\mathscr P$ is an {\sf MSO}-transduction of a class with bounded pathwidth, hence has bounded linear cliquewidth.
	
	Conversely, assume that $\mathscr C$ has bounded linear cliquewidth, by \zcref{th:main} it has a
    modelization with a coupling $\mathscr D$ of a class of bounded pathwidth colored
    graphs with a class of posets formed by unions of chains. 
	According to \zcref{lem:order_all}, there is a modelization of $\mathscr D$ in a class $\mathscr P'$ of colored posets, whose cover graphs form a weakly sparse class.
	As the class $\mathscr H$ of the cover graphs of the posets in  $\mathscr P'$ is a transduction of $\mathscr P'$, (hence of $\mathscr D$, thus of $\mathscr C$), it has bounded linear cliquewidth. According to \zcref{fact:ws}, as $\mathscr H$ is a weakly sparse class with bounded linear cliquewidth, it has bounded pathwidth.
	\end{proof}

We conjecture that this characterization  extends in a natural way to a characterization of classes with bounded cliquewidth.

\conjcw

We note that \zcref{thm:amalg_model} allows to give a slightly weaker decomposition than \zcref{lcw_pairing_poset} for classes with a parameter $2$ decomposition on a base class with bounded linear cliquewidth.

\begin{corollary}\label{lcw_dec_transeq}
If a class of graphs $\mathscr C$ has a bounded-size parameter 2 decomposition on a base class with bounded linear
cliquewidth graphs, then it has a modelization in a class of partially-ordered graphs $\mathscr D$, whose edge relation is degenerate
and whose poset has cover graphs of bounded pathwidth.
\end{corollary}

\begin{proof}
According to \zcref{lem:dec2split}, there exist integers $n$ and $h$ such that $\mathscr C$ admits splits of size $n$ and height $h$.
According to \zcref{thm:amalg_model}, $(\interp{SBuild}_{n,h},\sparse_{n,h})$ is a modelization of $\mathscr C$ in the class $\mathscr C^\text{mod}$ of its accessible models.
By construction, the $\vartriangleleft$-reduct of 
$\mathbf M\in \mathscr C^\text{mod}$ is a disjoint union of linear orders. On the other hand, the $E$-reduct of $\mathscr C^\text{mod}$ is a degenerate class.
\end{proof}

Graphs captured by \zcref{lcw_dec_transeq} and not by \zcref{th:main} include for
instance proper circular arc graphs. Degeneracy is however a rather weak property in
terms of logical tameness, as 1-subdivided cliques are 2-degenerate but
not even monadically dependent.
In order to strengthen the result of \zcref{lcw_dec_transeq}, we need to consider
decompositions for every parameter, which will be the aim of the next section.

\section{Decomposition of classes admitting bounded-size bounded linear cliquewidth decompositions}
\label{sec:lcw_dec}

In this section, we prove that the conclusion of \zcref{lcw_dec_transeq} can be strengthened when, instead of assuming that a class has a bounded-size parameter $2$ decomposition on a base class with bounded linear cliquewidth we assume that such a bounded-size decomposition exists for each value of the parameter.

\thmStrong

The main difficulties in the proof of this theorem are twofold:
First, the property of having bounded-size  bounded linear cliquewidth decompositions is not known to be preserved by transductions. 
Second, the model that we construct are not  well-preserved when we consider an induced subgraph: if $\mathbf M$ is a model of $G$ and $X\subseteq V(G)$, then $\mathbf M\langle X\rangle$ is generally not accessible.

Let us give an idea of the proof. 
As in \zcref{lcw_dec_transeq}, we first consider a bounded-size parameter $2$ decomposition of a graph $G$ on a base class with bounded linear cliquewidth and deduce the existence of a degenerate model $\mathbf M$.
Our aim is to prove that the so-obtained models actually admit bounded-size bounded linear cliquewidth covers (hence
bounded-size bounded linear cliquewidth
decompositions) for every value of the parameter.
This situation is fairly similar to that which occurs when considering the sparsification of a structurally bounded expansion class obtained from a bounded-size parameter $2$ bounded shrubdepth decomposition \cite{SBE_drops}.

In order to prove the property for a given value $p$ of the parameter, we consider a set $X$ of $p$ points in the domain of the models.
The first observation is that there exists a subset $Y_X$ of the ground of the model with $|Y_X|\leq|X|$ such that the model $\mathbf M\langle Y_X\rangle$ generated by $Y_X$ contains all the elements of $X$. 

In an ideal world, $\mathbf M\langle Y_X\rangle$ would be an accessible model of $G[Y_X]$, and we could derive from the fact that $G[Y_X]$
has bounded linear cliquewidth that $\mathbf M\langle Y_X\rangle$,
being a transduction of a graph with bounded linear cliquewidth, has itself bounded linear clique-width. 
However, it not true in general that $\mathbf M\langle Y_X\rangle$ is an accessible model of $G[Y_X]$.

In order to circumvent this issue, we prove that there exists a small set $W\supseteq Y_X$ (with $|W|\in\mathcal O(Y_X)$) such that 
an accessible model $\mathbf M_W$ of $G[W]$ has the same submodel induced by $Y_X$ as $\mathbf M$, that is $\mathbf M_W\langle Y_X\rangle=\mathbf M\langle Y_X\rangle$.

This motivates the introduction of the following notion of an anchor of an accessible model.
(The next definition is understood to apply for a given (fixed) class of graphs $\mathscr C$ and  (fixed) modelization  of $\mathscr C$.)

\begin{definition}
	    Let $\mathbf M,\mathbf M'$ be accessible models with respective grounds $L$ and $L'$, such that $L'\subseteq L$, and let $F:L\rightarrow\mathcal P(L)$ be a mapping.
	
	We say that $F$ is \emph{anchoring} for the pair $(\mathbf M, \mathbf M')$ if
    $u\in F(u)$ for every $u\in L$ and
	\[
	\forall X\subseteq L,\qquad\bigcup_{u\in X}F(u)\subseteq L'\quad\Longrightarrow\quad\mathbf M\langle X\rangle=\mathbf M'\langle X\rangle.
	\]
	(Formally, the equality $\mathbf M\langle X\rangle=\mathbf M'\langle X\rangle$ is understood here as the existence of an isomorphism of $\mathbf M\langle X\rangle$ and $\mathbf M'\langle X\rangle$ fixing $X$.)
	
	Let $\mathbf M$ be an accessible model with ground $L$.
	A mapping $F:L\rightarrow\mathcal P(L)$ is
    an \emph{anchor} of $\mathbf M$ if 
    for every  $L'\subseteq L$ there exists an accessible model $\mathbf T_{L'}$ with ground $L'$ such that $F$ is anchoring for $(\mathbf T,\mathbf T_{L'})$ (See \zcref{fig:enter-label}).
    The \emph{size} of the anchor $F$ is $\max_{x\in L}|F(x)|$.
\end{definition}

Note that $u\mapsto L$ is always an anchor, but the challenge is to produce an anchor of
bounded size.

\begin{figure}[ht]
    \centering
$\begin{xy}
    \xymatrix@R=.1\textwidth @C=.2\textwidth{
    \mathbf M\ar@/_/[dd]_{\interp{SBuild}_{n,h}}\ar[rr]^{\braket{X}}&&\mathbf M\langle X\rangle\ar[dd]^{\interp{SBuild}_{n,h}}\\
    &\mathbf M_W\ar[ur]^{\langle X\rangle}& \\
G\ar@{..>}@/_/[uu]_{\interp{Sparse}_{n,h}}\ar[r]^{[W]}&G[W]\ar@{..>}[u]_{\interp{Sparse}_{n,3h}}\ar[r]^{[X]}&G[X]
    }
\end{xy}
$
    \caption{We consider a clean  model $\mathbf M$ of a graph $G$.
    (Hence, $G=\interp{SBuild}_{n,h}(\mathbf M)$ and $\mathbf M\in\sparse_{n,h}(G)$.)
    Let $W\supseteq \bigcup_{v\in X}F(u)$. Then, applying $\sparse_{n,3h}$ to $G[W]$ is enough to recover the restriction $\mathbf M\langle X\rangle$ of $\mathbf M$.
    }
    \label{fig:enter-label}
\end{figure}

The next two subsections are dedicated to the construction of the model $\mathbf T_W$ in the case where $\mathbf T$ is a cotree and in the case 
where $\mathbf T$ is a bicotree.
The first case might have a more direct proof, but the given proof can be considered as a warm-up for the case of where $\mathbf T$ is a bicotree.

Recall that a cotree \emph{clean} if its internal nodes have at least two children and they alternate between joins and disjoint unions and that a cotree of height $h$ is clean if and only if it is accessible in the modelization  $(\build[1],\interp{Cotree}_h)$.

\begin{lemma}
	\label{lem:anchor1}
	Every clean cotree $\mathbf T$ of height at most $h$ has a anchor
    of size at most $h$.
\end{lemma}
\begin{proof}
	We prove the statement by induction on $h$.
	If $h=1$, then $F(u)=\{u\}$ is an anchor of $\mathbf T$. Let $h>1$ be such that the statement has been proved for every $h'<h$, and let $\mathbf T$ be a clean cotree of height $h$ and let $G=\build[1](\mathbf T)$. Let $r$ be the root of $\mathbf T$, let $L=L(\mathbf T)$, and let $v_1,\dots,v_k$ be the children of $r$. For $i\in[k]$, let $\mathbf T_i=T{\upharpoonright}v_i$, let $L_i=L(\mathbf T_i)$,  and let $F_i$ be an anchor of $\mathbf T_i$. (By induction, $F_i$ exists and has size at most $h-1$.)
	For $u\in L_i$, let $F(u)=F_i(u)\cup\{u'\}$, where $u'\in L$ and $u\curlywedge u'=r$.
	
	Let $L'\subseteq L$ and let, for $i\in[k]$, $L_i'=L'\cap L_i$. For each $i\in[k]$ such that $L_i'\neq\emptyset$, as $F_i$ is an anchor of $\mathbf T_i$, there exists a clean cotree $\mathbf T_{L_i'}$ such that $F_i$ is anchoring for
    $(\mathbf T_i,\mathbf T_{L_i'})$.
	
	Before constructing $\mathbf T_{L'}$ we take time for a short claim on the structure of induced subgraphs of a cographs.
	\begin{claim}
		\label{claim:cog}
		Let $\mathbf T$ be a clean cotree and let $G=\build[1](\mathbf T)$.
		
		Let $L'\subseteq L(\mathbf T)$. If $L'$ contains two vertices $x,y$, whose  infimum in $\mathbf T$ is the root of $\mathbf T$, then $G$ is connected if and only if $G[L']$ is connected.
	\end{claim}
	\begin{claimproof}
		Assume $G$ is connected. Then, the root of $\mathbf T$ is of type J. Hence, $x$ and $y$ are adjacent and every vertex $v\in V(G)$ is adjacent to $x$ or $y$.
		As $x$ and $y$ are both in $L'$, it follows that $G[L']$ is connected. 
		Conversely, assume that $G$ is disconnected. Then $\overline G$ is connected and, by the above, $\overline G[L']$ is connected. As $G[L']$ is a co-connected cograph, it is disconnected.
	\end{claimproof}
	
	We construct the cotree $\mathbf T_{L'}$ as follows: 
	\begin{itemize}
		\item If there is no $u\in L'$ such that $F(u)\subseteq L'$, we let $\mathbf T_{L'}$ be a clean cotree model of $G[L']$. 
		Then, $F$ is obviously anchoring for $(\mathbf T,\mathbf T_{L'})$.
		\item If the root of $\mathbf T$ is of type $U$, then we let the root $r$ of $\mathbf T_{L'}$ be of type $U$ and, for each $i\in[k]$ such that $L_i'\neq\emptyset$, 
		\begin{itemize}
			\item either the cotree $\mathbf T_{L'_{i}}$ is added as a child of $r$ if the type of its root is $J$,
			\item or the root of the cotree $\mathbf T_{L'_i}$ is identified with $r$ if its type is U.
		\end{itemize}
		By construction, the types of the internal nodes of $\mathbf T_{L'}$ alternate. 
		By assumption, there exists $u\in L'$ such that $F(u)\subseteq L'$. Hence, there exists $u'\in L'$ such that the infimum of $u$ and $u'$ in $\mathbf T$ is its root. By \zcref{claim:cog}, $G[L']$ is disconnected.
		Thus, the root $r$ of 	$\mathbf T_{L'}$ has at least two children. Thus, 	$\mathbf T_{L'}$
		is clean. 
		
		Let $X\subseteq L$ be such that $\bigcup_{u\in X}F(u)\subseteq L'$. Note that 
		$X\cap L_i=X\cap L_i'$ for every $i\in [k]$.
		If $u\in X\cap L_i'$, then $F_i(u)\subseteq L_i'$. Thus, as $F_i$ is anchoring for $(\mathbf T_i,\mathbf T_{L'_i})$, $\mathbf T_i\langle X\cap L_i\rangle=\mathbf T_{L_i'}\langle X\cap L_i\rangle$. Hence, as $\mathbf T\langle X\rangle$ (resp. $\mathbf T_{L'}$) is
        formed by joining $r$ to the roots of $\mathbf T_i\langle X\cap L_i\rangle$ (resp.
        $\mathbf T_{L'_i}\langle X\cap L_i\rangle$), $\mathbf T\langle X\rangle = \mathbf T_{L'}\langle X\rangle$ and $F$ is anchoring for $(\mathbf T, \mathbf T_{L'})$.
		\item If the root of $\mathbf T$ is of type $J$, the reasoning followed for the previous case applies \emph{mutatis mutandis}.
			\end{itemize}
\end{proof}

We now consider the bicotree case. Recall that every clean bicotree of height $h$ is  accessible in the modelization  $(\build[2],\interp{Bicotree}_h)$.
\begin{lemma}
	\label{lem:anchor2}
    Every clean bicotree $\mathbf T$ of height at most $h$ has an anchor $F$ of size less than $5^{h+1}$.
\end{lemma}
\begin{proof}
    We prove the statement by induction on the height $h$ of $\mathbf T$.
    If $h=1$, then $F(u)=\{u\}$ is an anchor of $\mathbf T$.
    Let $h>1$ be such that the statement has been proved for every $h'<h$, and let $\mathbf T$ be a clean bicotree with height $h$.
    Let $r$ be the root of $\mathbf T$, let $L=L(\mathbf T)$, and let $v_1,\dots,v_k$ be the children of $r$. For $i\in[k]$, let $\mathbf T_i=\mathbf T{\upharpoonright} v_i$, let $L_i=L(\mathbf T_i)$, let $G_i=\build[2](\mathbf T_i)$, and let  $F_i$ be an anchor of $\mathbf T_i$. (By induction, $F_i$ exists and has size less than $5^{h-1}$.)

    We consider different cases depending on the type of $r$.
\begin{itemize}
    \item Assume $r$ is of type U. We choose $t_i\in L_i$ for each $i\in[k]$. Note that each $G_i$ is connected as $\mathbf T$ is clean.
    For $u\in L_i$, define $F(u)$ as the union, for $x\in F_i(u)$, of the vertex set of a shortest path from $x$ to $t_i$. According to \zcref{sob_p7}, these paths have length at most $5$, so $|F(u)|\leq 5|F_i(u)|+1$. 

    Let $L'\subseteq L$. We construct the bicotree $\mathbf T_{L'}$ as follows: For each $i$ with $t_i\in L'$, let $H_i$ be the connected component of $G_i[L']$ containing $t_i$ and let $L_i'=L'\cap V(H_i)$.
    As $F_i$ is an anchor of $\mathbf T_i$, there exists
    a clean bicotree $\mathbf T_{L_i'}$ with height at most $3(h-1)$ such that $F_i$ is
    anchoring for $(\mathbf T_i,\mathbf T_{L_i'})$.
    We construct the bicotree $\mathbf T_{L'}$ as follows: it has a root $r$ of type U, and the children of $r$ are all the $\mathbf T_{L_i'}$ (for $i$ such that $t_i\in L'$) and a clean  bicotree model with height at most $3(h-1)$ for each connected component of $G[L']$ that contains no $t_i$ (such a clean bicotree exists, according to \zcref{lem:cleanbic}). By construction, $\mathbf T_{L'}$ is clean and has height at most $3h-2$.

    Let $X\subseteq L$ be such that $\bigcup_{u\in X}F(u)\subseteq L'$. If $u\in L_i$, then as $G_i[F(u)]$ is connected and contains $t_i$, we have $F(u)\subseteq L_i'$. Hence, for every $i\in[k]$ such that $X\cap L_i\neq\emptyset$, we have $\bigcup_{u\in X\cap L_i}F_i(u)\subseteq\bigcup_{u\in X\cap L_i}F(u)\subseteq L_i'$. As $F_i$ is anchoring for
    $(\mathbf T_i,\mathbf T_{L_i'})$, we deduce $\mathbf T_i\langle X\cap L_i\rangle=\mathbf T_{L_i'}\langle X\cap L_i\rangle$. It follows that $\mathbf T\langle X\rangle=\mathbf T_{L'}\langle X\rangle$, hence $F$ is an anchor of $\mathbf T$.
    \item Assume $r$ is of type B. This case is handled like the case where the root is of type U  (considering connected components of the bi-complement).
    \item Assume $r$ is of type O.
    For each $i\in [k]$ and $j\in\{1,2\}$, let $t_{i,j}\in L_i$ have color $j$. (Such vertices exist as $\mathbf T$ is clean.) For $u\in L_i$ we define
\[
F(u)=F_i(u)\cup \{t_{i',j}\colon i'\in\{1,k\}, j\in\{1,2\}\}\cup \{t_{i',j}\colon |i-i'|\leq 1, j\in\{1,2\}\}.
\]
    Note that $|F(u)|\leq |F_i(u)|+10$.
    
    Let $L'\subseteq L$ and let, for $i\in[k]$, $L_i'=L'\cap L_i$. For each $i\in [k]$, as $F_i$ is an anchor of $\mathbf T_i$, there exists a clean bicotree $\mathbf T_{L_i'}$ with ground $L'_i$ such that $F_i$ is anchoring for 
    $(\mathbf T_i,\mathbf T_{L_i'})$.
    We construct the bicotree $\mathbf T_{L'}$ as follows: 
    \begin{itemize}
        \item If $L_1'$ or $L_k'$ is monochromatic, $\mathbf T_{L'}$ is a clean  bicotree model of $G[L']$ with height at most $3h$ (such a clean bicotree exists, according to \zcref{lem:cleanbic}).
        Indeed, no $u\in L'$ will ever satisfy $F(u)\subseteq L'$, so there is nothing to check.
        \item Otherwise, there exists a partition of the set $\{i\in[k]\colon L_i'\neq\emptyset\}$ into successive intervals $I_1,\dots,I_\ell$ and indices $p_s\in I_s$ (for $s\in[\ell])$, such that
        
        \begin{itemize}
            \item $L_i'$ is monochromatic with color $1$ if $i\in I_s$ and $i<p_s$;
            \item $L_i'$ is monochromatic with color $2$ if $i\in I_s$ and $i>p_s$;
            \item $\bigcup_{i\in I_s}L'_i$ is bichromatic.            
        \end{itemize}
        The bicotree $\mathbf T_{L'}$ has a root $r$ of type O and, for each $s\in [\ell]$ (taken in order),
        \begin{itemize}
        \item $\mathbf T_{L_i'}$ as a child if $I_s=\{i\}$,
        \item or a child of type B having the bicotrees $\mathbf T_{L_i'}$ with $i\in I_s$ as children.
        \end{itemize}
        By construction, $\mathbf T_{L'}$ is a clean bicotree with height at most $3h-1$
        and models $G[L']$.

        Let $X\subseteq L$ be such that $\bigcup_{u\in X}F(u)\subseteq L'$. Note that $X\cap L_i=X\cap L_i'$ for every $i\in[k]$.
        Let $i\in [k]$ such that $X\cap L_i' \neq \emptyset$. As $F_i$ is anchoring for $(\mathbf T_i,\mathbf T_{L_i'})$, $\mathbf T_i\langle X\cap L'_i\rangle=\mathbf T_{L_i'}\langle 
        X\cap L'_i\rangle$. Let $s\in [\ell]$ such that $i\in I_s$.
        As none of the sets $L_{i-1}'$ and $L_{i+1}'$ may be monochromatic (by our choice of $F$), we get $|I_s|=1$, so $\mathbf T_{L_i'}$ is a child of $r$. We deduce that
        $\mathbf T\langle X\cap L_i\rangle = \mathbf T_{L'_i}\langle X\cap L_i\rangle$,
        and then
        $\mathbf T\langle X\rangle=\mathbf T_{L'}\langle X\rangle$. Thus,
        $F$ is an anchor of $\mathbf T$.\qedhere
     \end{itemize}
\end{itemize} 
\end{proof}

\begin{proof}[Proof of \zcref{strong_thm}]
    Let $\mathscr C$ be a  class with bounded-size bounded linear cliquewidth decompositions.
    According to  \zcref{lem:dec2split},
    there are integers $n,h$ such that every graph $G\in\mathscr C$ a split 
    of size $n$ an height $h$.
    According to \zcref{thm:amalg_model}, the pair $(\interp{SBuild}_{n,h},\sparse_{n,h})$ is a modelization  of $\mathscr C$ in amalgams. Let $\mathscr C^\text{mod}$ be the associated class of accessible models. As noted in the proof of \zcref{lcw_dec_transeq}, every $\mathbf M\in\mathscr C^\text{mod}$ is a coupling of a degenerate class of graphs and of a class of posets consisting in disjoint unions of chains.

    We now aim to prove that $\mathscr C^\text{mod}$ has bounded-size bounded linear cliquewidth covers, hence has an 
    $E$-reduct with bounded expansion (as the $E$-reduct is degenerate).

    Let $G\in\mathscr C$, and let $\mathbf M$ an accessible model of $G$ in $\mathscr C^\text{mod}$. By construction, there exists a split $\gamma_G$ of $G$ of size $n$ and height at most $h$, clean cotree models $\mathbf T_i$ of $G[\gamma_G^{-1}(i)]$ and clean bicotrees models $\mathbf T_{i,j}$ of $G[\gamma_G^{-1}(i),\gamma_G^{-1}(j)]$ such that 
    $\mathbf M=\am{V(G),(\mathbf T_i),(\mathbf T_{i,j})}$.
    Let $(\iota,(\iota_i),(\iota_{i,j}))$ be an injection system of $\mathbf M$.
    
    Let $X$ be a subset of (at most) $p$ vertices of $\mathbf M$. 
    For each $x\in X$, select $y(x)\in \ground(\mathbf M)$ such that $x\in \mathbf M\langle \{y(x)\}\rangle$ (consider
    cases depending on whether $x \in \ground(\mathbf M)$, $x\in \mathbf T_i$ or $x\in \mathbf T_{i,j}$
    for some $i,j \in [n]$).
    Let $Y(X)=\{y(x)\colon x\in X\}$.
    Then,  
    $|Y_X|\leq p$ and $X\subseteq \mathbf M\langle Y_X\rangle$. 
    According to \zcref{lem:anchor1,lem:anchor2},
    each $\mathbf T_i$ has an anchor
     $F_i$ with size at most $h$ and 
     each  $\mathbf T_{i,j}$ has an anchor $F_{i,j}$ with size less than $5^{h+1}$. 
     We define $F:V(G)\rightarrow\mathcal P(V(G))$ by
\[
F(x)=F_{\gamma_G(x)}(x)\cup\bigcup_{j\neq\gamma_G(x)}F_{\gamma(x),j}(x).
\]
Let $q:=\max_{x\in V(G)}|F(x)|$. Note that
$q<h+5^{h+1}(n-1)$.

As $\mathscr C$ admits bounded size bounded linear cliquewidth decompositions, there exist, for each integer $p$, integers $m_p,t_p$ such that 
every graph $G\in\mathscr C$ has a coloring
$\zeta:V(G)\rightarrow [m_p]$ such that the union of any $pq$ color classes has linear cliquewidth at most $t_p$. (Note that we use a decomposition with parameter $pq$ and not $p$.)

For $S\subseteq [m_p]$ with $|S|\leq pq$, let 
$U_S=\{u\in V(G)\colon \zeta(F(u))\subseteq S\}$.
We denote by $\gamma_S$ the restriction of
$\gamma_G$ to $U_S$.
Let $i\in [n]$, as $F_i$ is an anchor of $\mathbf T_i$ there exists a clean cotree $\mathbf T_i'$
with set of leaves $\gamma_S^{-1}(i)$
such that $F_i$ is anchoring $(\mathbf T_i,\mathbf T_i')$. (Hence, $L(\mathbf T_i')=L(\mathbf T_i)\cap U_S$.)
Similarly, there exists a clean bicotree $\mathbf T_{i,j}'$ with set of leaves $\gamma_S^{-1}(i)\cup\gamma_S^{-1}(j)$
such that $F_{i,j}$ is anchoring $(\mathbf T_{i,j},\mathbf T_{i,j}')$. 
Let $\mathbf R_S=\am{U_S,(\mathbf T_i'),(\mathbf T_{i,j}')}$. Note that  $\mathbf R_S\in\sparse_{3h}(G[U_S],\gamma_S)$.

Now, let $X$ be a set of $p$ points in $\mathbf M$. 
According to the definition of $\zeta$, there exists $S\in[m_p]$ such that $|S|\leq pq$ and 
$\zeta(F(Y_X))\subseteq S$ (Hence, such that
$Y_X\subseteq U_S$). It follows that 
$\mathbf T_i'\langle Y_X\cap \gamma_S^{-1}(i)\rangle=\mathbf T_i\langle Y_X\cap \gamma_S^{-1}(i)\rangle$ and 
$\mathbf T_{i,j}'\langle Y_X\cap (\gamma_S^{-1}(i)\cup\gamma_S^{-1}(j))\rangle=\mathbf T_{i,j}\langle Y_X\cap (\gamma_S^{-1}(i)\cup\gamma_S^{-1}(j))\rangle$. Thus,
$\mathbf R_S\langle Y_X\rangle=\mathbf M\langle Y_X\rangle$. It follows that 
$\mathbf M[X]\subseteq_i \mathbf M\langle Y_X\rangle=\mathbf R_S\langle Y_X\rangle\subseteq_i \mathbf R_S\langle U_S\rangle$.
 The structures $\mathbf R_S\langle U_S\rangle$ belong to a class $\mathscr S$ that is a transduction of a class with linear cliquewidth at most $t_p$, thus has bounded linear cliquewidth. Since $\mathbf R_S\langle U_S\rangle = \mathbf M_S\langle U_S\rangle$, they form a  bounded-size bounded linear cliquewidth cover of
 $\mathbf M$.
 It follows that the class $\mathscr C^\text{mod}$ has bounded-size bounded linear cliquewidth covers, hence bounded-size bounded linear cliquewidth decompositions.

\end{proof}
\section{Concluding remarks}
\label{sec:conc}
One common paradigm in the study of general dependent structures and monadically dependent classes of graphs is the appearance of tree-like structures. For graphs, those of bounded treedepth, shrubdepth, treewidth, cliquewidth, twin-width and merge-width can all be
generated to some extent from tree models, enriched with colors (unary predicates) or stable
unordered relations. In classical model theory, the decomposition of types into a stable part and a tree-like quotient was used by Shelah
\cite{shelah2015dependent} to prove the so-called generic pair conjecture in dependent theories
(see also \cite{simon2015guide} for more discussions on this topic).

In this line, we conjecture the following description of hereditary dependent classes of graphs.

\begin{conjecture}
	\label{conj:main}
	Every dependent hereditary class of graphs is a transduction of a dependent
	hereditary class, consisting of a nowhere dense class of graphs expanded by a tree order.
\end{conjecture}

Note that Conjecture~\ref{conj:ND} implies, if true,  that Conjecture~\ref{conj:main}
holds for stable hereditary classes of graphs. Also, it is known that
Conjecture~\ref{conj:main} holds for classes with bounded twin-width \cite{SparseTww} (which have a particular model-theoretic meaning, as recalled above).

Dual to Conjecture~\ref{conj:main}, it is plausible that the graphs in a dependent hereditary class of graphs can be, in some sense, decomposed into a stable part and an ordered part.
This intuition echoes the result proved by Simon \cite{simon2019type}, which shows that dependent types decompose into  a stable part and an order-like quotient. 

Such decomposition results would rely on two notions,
 \emph{transduction pairing}, which is a strong notion of transduction equivalence
 between classes of structures, and \emph{class coupling}, which expresses that a
 structure is obtained by putting together relations from simpler structures (See formal
 definitions in \zcref{sec:trans} and \zcref{sec:coupling}).

In this setting we conjecture the following.

\conjMain

Note that posets whose cover graph have bounded treewidth are transductions of tree orders (as follows from~\cite{Colcombet2007})
and have nice properties. For instance, they have  bounded Boolean dimension \cite{Felsner2020} and bounded cliquewidth\footnote{If the cover graph has bounded treewidth, then it is an {\sf MSO}-transduction of a tree. As the poset is an easy {\sf MSO}-transduction of the cover graph, we deduce that the poset is an {\sf MSO}-transduction of a tree, hence has bounded cliquewidth.}. %
 It follows that Conjectures~\ref{conj:ND}~and~\ref{conj:inverse} together imply Conjecture~\ref{conj:main}.

First, we notice that the conjecture has to impose some restriction on the poset, for otherwise it would obviously hold: every class of graphs has a modelization in a class of posets. Indeed, we consider the transduction $\interp{T}$ that is the composition of a copy operation that creates $3$ clones $(v,1), (v,2)$, and $(v,3)$ of each vertex $v$, then define (for distinct $(u,i)$ and $(v,j)$)
\[
(u,i)<(v,j)\quad\iff\quad\begin{cases}
    j=3\text{ and }u=v\\
    \text{ or }i=1,j=2,\text{ and }E(u,v).
\end{cases}
\]
As the constructed posets have bounded height, they are transduction-equivalent to their cover graphs. Hence,
if we apply this transduction to a hereditary unstable class of graphs, the cover graphs of the obtained posets will not be stable.
This shows that the weakest non-trivial version of \zcref{conj:inverse} would be
the following.

\begin{conj}
	\label{conj:inverseweak}
	Every dependent hereditary class of graphs has a modelization in a dependent hereditary class, consisting in the expansion of a hereditary stable class of graphs by a partial order having a stable cover graph.
\end{conj}

We start with an interesting consequence of Conjecture~\ref{conj:inverse}, if true.

\begin{proposition}
	 \zcref{conj:inverse} implies \zcref{conj:cw}.
\end{proposition}
\begin{proof}
	If $\mathscr C$ is transduction-equivalent to a class $\mathscr P$ of posets whose cover graphs have bounded treewidth, then $\mathscr P$ is an {\sf MSO}-transduction of a class with bounded treewidth, hence has bounded cliquewidth.
	Conversely, assume that $\mathscr C$ has bounded cliquewidth.
	Then,  assuming \zcref{conj:inverse}, $\mathscr C$ is transduction equivalent to a $\varsigma$-structure, whose $E$-reduct is stable and whose $<$-reduct is a class of posets with bounded treewidth.
	Then, the $E$-reduct is  transduction-equivalent to a class $\mathscr D$ with bounded expansion \cite{SBE_TOCL}. Considering $\mathscr D$ instead of $\mathscr C$ if necessary, we may assume that the $E$-reduct is biclique-free.
	Then, according to \zcref{lem:order_all}, there exists a modelization of
	$\mathscr C$ in a class of posets whose cover graphs form a weakly sparse
	class $\mathscr H$. A $\mathscr H$ is (by transitivity) a transduction of a class with bounded cliquewidth, it has bounded cliquewidth. Then, it follows from \zcref{fact:ws} that $\mathscr H$ has bounded treewidth. %
\end{proof}

\bibliography{ref}
\end{document}